\documentclass{elsarticle}
\usepackage{hyperref}

\usepackage{indentfirst,latexsym,bm,amsmath,graphicx,amsthm,amssymb}
\usepackage{color}

\usepackage{framed} 
\usepackage{paralist} 
\usepackage{multirow} 
\usepackage{verbatim} 
\usepackage{array}
\usepackage{booktabs}  

\usepackage{mathrsfs}

\usepackage{graphicx}
\usepackage{subcaption}
\usepackage{epstopdf}
\newtheorem{theorem}{Theorem}[section]
\newtheorem{definition}{Definition}
\newtheorem{corollary}{Corollary}
\newtheorem{proposition}{Proposition}
\newtheorem{lemma}{Lemma}

\newcommand{\C}{\mathcal{C}}

\renewcommand{\S}{\mathcal{S}}
\newcommand{\U}{\mathcal{U}}

\newcommand{\tz}{{\tilde{z}}}
\newcommand{\fp}{{\mathfrak{p}}}
\newcommand{\fP}{{\mathfrak{P}}}

\newcommand{\rlambda}{{\bar{\lambda}}}
\newcommand{\rmu}{{\bar{\mu}}}

\begin{document}

\begin{frontmatter}

\title{FCFS Parallel Service Systems and Matching Models}

\author[address1]{Ivo Adan}
\ead{iadan@win.tue.nl}

\author[address2]{Igor Kleiner\fnref{thanks}}
\ead{igkleiner@gmail.com}

\author[address3]{Rhonda Righter}
\ead{rrighter@ieor.berkeley.edu}

\author[address2]{Gideon Weiss\corref{mycorrespondingauthor}\fnref{thanks}}
\cortext[mycorrespondingauthor]{Corresponding author}
\ead{gweiss@stat.haifa.ac.il}

\address[address1]{Eindhoven University of Technology}
\address[address2]{Department of Statistics,  
The University of Haifa, Mount Carmel 31905, Israel}
\address[address3]{University of California at Berkeley}

\fntext[thanks]{Research supported in part by
Israel Science Foundation Grant 286/13.}

\begin{abstract}
We consider three parallel service models in which customers of several types are served by several types of servers subject to a bipartite compatibility graph, and the service policy is first come first served.  Two of the models have a fixed set of servers.  The first is a queueing model in which arriving customers are assigned to the longest idling  compatible server if available, or else queue up in a single queue, and servers that become available pick the longest waiting compatible customer, as studied by Adan and Weiss, 2014.  The second is a redundancy service model where arriving customers  split into copies that queue up at all the compatible servers, and are served in each queue on FCFS basis, and leave the system when the first  copy completes service, as studied by Gardner et al., 2016.  The third model is a matching queueing model with a random stream of arriving servers.  Arriving  customers queue in a single queue and arriving servers  match with the first compatible customer and leave immediately with the customer, or they leave without a customer.  The last model is relevant to organ transplants, to housing assignments, to adoptions and many other situations.

We study the relations between these models, and show that they are closely related to the FCFS infinite bipartite matching model, in which two infinite sequences of customers and servers of several types are matched FCFS according to a bipartite compatibility graph, as studied by  Adan et al., 2017.  We also introduce a directed  bipartite matching model in which we embed the queueing systems.  This leads to a generalization of Burke's theorem to parallel service systems.
\end{abstract}


\end{frontmatter}


\section{Introduction}
We consider three parallel service models in which customers of several types, indexed by $c_i\in \C =\{c_1,\ldots,c_I\}$  are served by several types of servers, indexed by $s_j\in \S=\{s_1,\ldots,s_J\}$,   subject to a bipartite compatibility graph,  $\mathcal{G}=(\S,\C,\mathfrak{E})$, $\mathfrak{E} \subseteq \S\times \C$, such that $(s_j,c_i)\in\mathfrak{E}$ if customer type $c_i$ can be served by server type $s_j$.  We focus on first come first served  (FCFS) policy in all the models, i.e. customers are prioritized by their order of arrivals, and servers are prioritized by the order in which they become available.  Two of the models have a fixed set of servers, while the third model has a random stream of arriving servers.
Briefly stated the models are as follows:

\begin{compactitem}[-]
\item
{\em FCFS-ALIS Parallel Queueing Model: }  
There are $J$ servers of types $\S$ and a stream of customers of types $\C$. An arriving customer is assigned to the longest idle server which is compatible with it (ALIS - assign longest idle server) if such is available, or else he joins the queue of waiting customers.  A server that completes a service picks up the longest waiting customer which is compatible with him (FCFS), if such is available, or else he joins the queue of idle servers.  This model was  studied by Adan and Weiss \cite{adan2014skill}.
  \item
{\em A Redundancy Service Model: }  
There are $J$ servers of types $\S$, each with his own FCFS queue, 
 and a stream of arriving customers of types $\C$. An arriving customer splits upon arrival into several copies
 that join the queues of the servers which are compatible with it.  Service of a customer can then proceed simultaneously at several compatible servers.  The customer and all its copies leave the system when the first of its copies completes service.
 This model was  studied by Gardner et al. \cite{gardner2016queueing}.  
  \item
{\em A  Parallel FCFS Matching Queue: }  
There is an arrival stream of customers of types  $\C$, and an independent arrival stream of servers of types $\S$.    When a customer arrives he joins a queue of customers waiting for service.  When a server arrives he scans the queue of customers and matches with the longest waiting customer that is compatible with his type, and the matched customer then leaves the system with the server.  If the server does not find a match he leaves immediately without a match.

The matching queue model  is relevant to many types of service systems:  It can describe organ transplants, where patients are waiting to receive organs, and  donated organs arrive in a random stream, and organs are assigned to compatible recipients in FCFS order, or are lost if no compatible recipient is waiting \cite{su2005patient}.  It can also describe an adoption process, where families are waiting for available babies to be adopted (this may only be approximate since unmatched babies do not disappear).  It was used to model assignment of project houses  to families in Boston public housing, by Kaplan \cite{kaplan1984managing,kaplan1988public}.  
Another application is to call centers with inbound and outbound calls, where differently skilled agents (servers) start outbound calls if there are no waiting inbound calls that match their skill sets. Here the state would be the set of customers waiting in the queue, and would not include those in service. 
Our matching queue model, although it seems very relevant to the study of organ transplants and to various other systems, has not, to the best of our knowledge, been analyzed in any level of detail.
\end{compactitem}

We assume Poisson arrivals and exponential server dependent service times for all three models so that their evolution is Markovian and can be described by various discrete-space continuous-time Markov chains.

These models are closely related to a fourth model:
\begin{compactitem}[-]
\item
{\em The  FCFS infinite bipartite matching model:}   This was introduced in \cite{caldentey2009fcfs,adan2012exact} and studied in more detail recently by Adan, Busic, Mairesse and Weiss \cite{adan2015reversibility}.  In this model there are two infinite sequences, drawn independently, one is drawn i.i.d. from $\C$, the other from $\S$, and the two sequences are then matched FCFS according to the compatibility graph $\mathcal{G}$.  This model is much simpler than either of the above models since it does not involve arrival times and service times, and servers and customers play a completely symmetric role.  
\end{compactitem}

In this paper we explore the relations between the three service models, and their connections to the FCFS infinite matching model.  Our results here are:
\begin{compactitem}[-]
\item
The continuous-time Markov chains that describe all three service models share the same stationary distribution.  This leads the way to comparing their performance measures.
\item
We note that the redundancy service model and the matching queue are equivalent, in that they share the same continuous-time Markov chain.  
\item
We compare the performance of the Redundancy Service model and the FCFS-ALIS model, and point out when either should be preferred. 
\item
In particular we study their performance for the `N'-system, and obtain sharp stochastic bounds on the difference in the number in queue for each policy. 
\item
We introduce a new discrete FCFS infinite matching model, which we call the FCFS infinite directed  matching model, that is similar to the model of \cite{adan2015reversibility}.  
\item
We derive properties of this new FCFS infinite directed matching model. 
\item
We embed the three service models in the infinite directed bipartite matching model.
\item
We obtain a version of Burke's Theorem for the redundancy service and for the matching queue systems.
\end{compactitem}

The rest of the paper is structured as follows:  In Section \ref{sec.servicemodels} we describe the three  models, and in Section \ref{sec.relations} we compare their performance.  In Section \ref{sec.N-model} we study performance of the `N'-system under FCFS-ALIS and under Redundancy Service, and present computational and simulation results.  
 In Section \ref{sec.infmatching} we describe the relevant properties of the FCFS infinite bipartite matching model.   In Section \ref{sec.embedded} we introduce the new FCFS infinite directed matching model, and derive properties of the process.
  In Section \ref{sec.surprise} we show how to  embed the three service models in this new matching model, and discover some surprising consequences of this embedding.
We complete the proofs of our results in appendices.

\subsection*{Notation}
We let: $\S(c_i)$ denote the subset of server types compatible with $c_i$, $\C(s_j)$ denote 
the subset of customer types compatible with $s_j$.   For  $C \subset \C$, $S\subset \S$ we let $\S(C) =\bigcup_{c_i\in C} \S(c_i)$, $\C(S) =\bigcup_{s_j\in S} \C(s_j)$,  and denote by $\U(S) = (\C(S^c))^c$ those customer types that are compatible only with server types in $S$.

We associate with $c_i$ a rate $\lambda_{c_i}$, and with $s_j$ a rate $\mu_{s_j}$, these are rates for exponential distributions.   We also let $\rlambda =  \sum_{i=1}^I \lambda_{c_i}$, $\rmu = \sum_{j=1}^J \mu_{s_j}$.
For  subsets $C \subset \C$, $S\subset \S$ we let $\lambda_C=\sum_{c_i \in C} \lambda_{c_i}$,
$\mu_S=\sum_{s_j \in S} \mu_{s_j}$.

In what follows we will denote quantities related to the FCFS-ALIS model by a superscript $^q$,
those related to the Redundancy Service model by a superscript $^r$,  those related to the Matching model by a superscript $^m$.  In addition, we denote quantities related to the FCFS infinite bipartite matching model by a superscript $^\infty$, and those related to the FCFS infinite directed  matching model by a superscript $^{\downarrow\infty}$.

\section{The Service Models}
\label{sec.servicemodels}
\subsection{A stability condition}
\label{sec.stabilitycondition}
\begin{theorem}
All three service models are stable, in the sense that Markov chains describing them are ergodic, if and only if the following condition holds:
\begin{equation}
\label{eqn.stability}
\lambda_C < \mu_{S(C)},  \qquad \mbox{for every }  C\subseteq \C.
\end{equation}
\end{theorem}
\begin{proof}  
This follows from the form of the solutions to the balance equations, that converge if and only if (\ref{eqn.stability}) holds.
\end{proof}
Figure \ref{fig.compatiiblity}  illustrates the compatibility graph for an example we will use throughout  the paper.  In this example there are 3 types of customers and 3 types of servers, customers of type $c_2$  (type $c_3$) can only be served by server of type $s_2$ (type $s_3$), while customers of type $c_1$ can be served by all types of servers.  This model is referred to in the literature as the `W'-model.
\begin{figure}[htbp]
   \centering
   \includegraphics[width=1.1in]{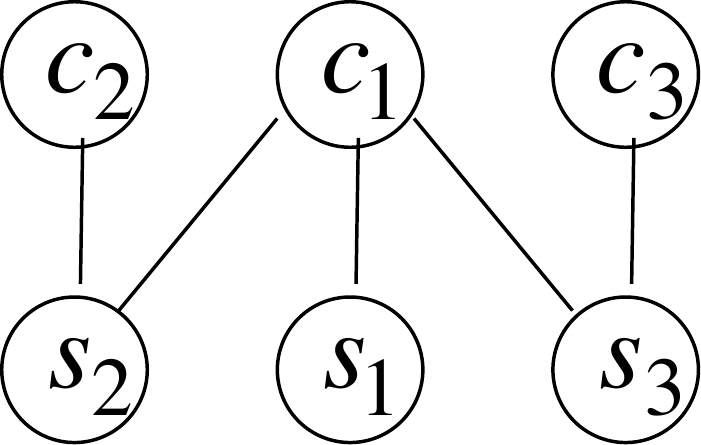} 
   \caption{Compatibility graph for customer and server types}
   \label{fig.compatiiblity}
\end{figure}

The stability condition for this example is: 
\[
\lambda_2 < \mu_2, \quad \lambda_3<\mu_3,  \quad \rlambda < \rmu.
\]

\subsection{The FCFS-ALIS parallel queueing model}
\label{sec.fcfsalis}
Customers arrive in independent Poisson streams, with rate $\lambda_{c_i}$ for type $c_i$.  There are $J$ servers of types $\{s_1,\ldots,s_J\}$, and service by server $s_j$ is exponential  with rate $\mu_{s_j}$.  The service policy as described in the introduction is FCFS-ALIS.  Figure \ref{fig.fcfsalis}  illustrates a possible state for our example.  
\begin{figure}[htbp]
   \centering
   \includegraphics[width=3.0in]{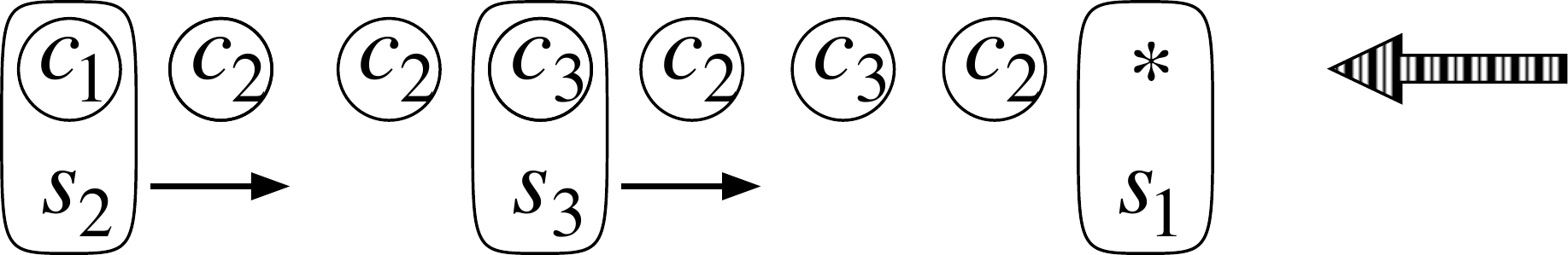} 
   \caption{A current state under FCFS-ALIS}
   \label{fig.fcfsalis}
\end{figure}
In this figure all customers in the system are displayed in order of arrival, with earlier arrivals more to the left.  Customers in service are shown together with their server.  The oldest customer in the system is of type $c_1$ and it is served by server $s_2$, server $s_3$ is serving a customer of type $c_3$ after skipping two incompatible customers of type $c_2$.  Server $s_1$ is idle.  In the future, new customers will  arrive from the right and join the end of the queue, with or without a compatible server, and on completion of service servers move to the right and scan waiting customers until they found a compatible customer or join the end of the idle servers queue.

In \cite{adan2014skill} the system  is described by the process $Y^q(t)=(S_1,n_1$, $\ldots,S_i,n_i,S_{i+1},\ldots,S_J)$ where $S_1,\ldots,S_J$ is a permutation of the servers, servers $S_1,\ldots,S_i$ are busy with $S_1$  serving the oldest customer in the system, $S_2$ has skipped $n_1$ customers and is serving the second oldest customer currently in service, and so on.  $n_j$ is the number of skipped customers between $S_j$ and $S_{j+1}$. The remaining servers, $S_{i+1},\ldots,S_J$ are idle, ordered by length of time they were idle, with $S_J$ the longest idle.  This describes the system at time $t$.  They proved:
\begin{theorem}[Adan and Weiss \cite{adan2014skill}]
\label{thm.adanweiss}
The process $Y^q(t)$ is a continuous-time discrete state Markov chain.  It is ergodic if and only if (\ref{eqn.stability}) holds.  Its stationary distribution is given, up to a normalizing constant, by 
\begin{eqnarray}
\nonumber
&P^q(S_1,n_1,\ldots,S_i,n_i,S_{i+1},\ldots,S_J) \propto 
\prod_{j=1}^i \frac{(\lambda_{\U(\{S_1,\ldots,S_j\})})^{n_j}}{(\mu_{\{S_1,\ldots,S_j\}})^{n_j+1}}  
\\
\label{eqn.adanweiss}
&\times \prod_{j=i+1}^J \frac{1}{\lambda_{\C(\{S_j,\ldots,S_J\})}}
\end{eqnarray}
\end{theorem}
Adan and Weiss \cite{adan2014skill} also calculated the normalizing constant.

We introduce an alternative process to describe the system,  $X^q(t) = (c^1,c^2,\ldots,c^L,s^1,\ldots,s^K)$ where $c^\ell$ is the random type of the $\ell$th oldest customer  in the system that is waiting and has not started service yet, and $s^k$  is the type of the $k$th longest idling server in the system, all this at time $t$.  Note that $L$, the number of waiting customers corresponds to $n_1+\cdots+n_i$ of $Y^q(t)$, and can take any value $\ge 0$, while $s^1,\ldots,s^K$ correspond to $S_J,\ldots,S_{i+1}$ of $Y^q(t)$, which are the ordered subset of idle servers, with no replications, so that $K\le J$.
We then have:
\begin{theorem}
\label{thm.fcfsalis}
The process $X^q(t)$ is a continuous-time discrete state Markov chain.  It is ergodic if and only if the stability condition (\ref{eqn.stability}) holds.  Its stationary distribution is given, up to a normalizing constant, by:
\begin{eqnarray}
\nonumber
& P^q(c^1,c^2,\ldots,c^L,s^1,\ldots,s^K) \propto 
\prod_{\ell=1}^L \frac{\lambda_{c^\ell}}{\mu_{\S(\{c^1,\ldots,c^\ell\})}}  \\
\label{eqn.fcfsalis} 
&\times \prod_{k=1}^K \frac{\mu_{s^k}}{\lambda_{\C(\{s^1,\ldots,s^k\})}} 
\end{eqnarray}
\end{theorem}
The proof of this theorem is by partial balance, we present it in   \ref{sec.partialbalance}.
In particular, the following corollary is immediate:
\begin{corollary}
\label{thm.fcfsalisbusy}
The process $X^q(t)$  conditional on the event that all servers are busy, has the stationary distribution given up to a normalizing constant by:
\begin{eqnarray}
\label{eqn.fcfsalisbusy} 
& P^q(c^1,c^2,\ldots,c^L\,|\, \mbox{all busy}) \propto 
\prod_{\ell=1}^L \frac{\lambda_{c^\ell}}{\mu_{\S(\{c^1,\ldots,c^\ell\})}}  
\end{eqnarray}
\end{corollary}

\subsection{The redundancy service model}
\label{sec.redundancy}
There are servers $s_1,\ldots,s_J$, and each of them has his own FCFS queue of compatible customers, and service by server $s_j$ is exponential with rate $\mu_{s_j}$.  Customers arrive in independent Poisson streams, with rate $\lambda_{c_i}$ for customers of type $c_i$.   Each arriving customer, upon arrival, splits into copies of the same type, and one copy joins the queue of each of the servers with which it is compatible.  Service of a customer can then be performed at several compatible servers simultaneously.  The customer departs from the system, with all its copies, at the instant at which service of one of its copies is complete.  

Figure \ref{fig.redundant}  illustrates a possible state for our example.  
\begin{figure}[htbp]
   \centering
   \includegraphics[width=3.10in]{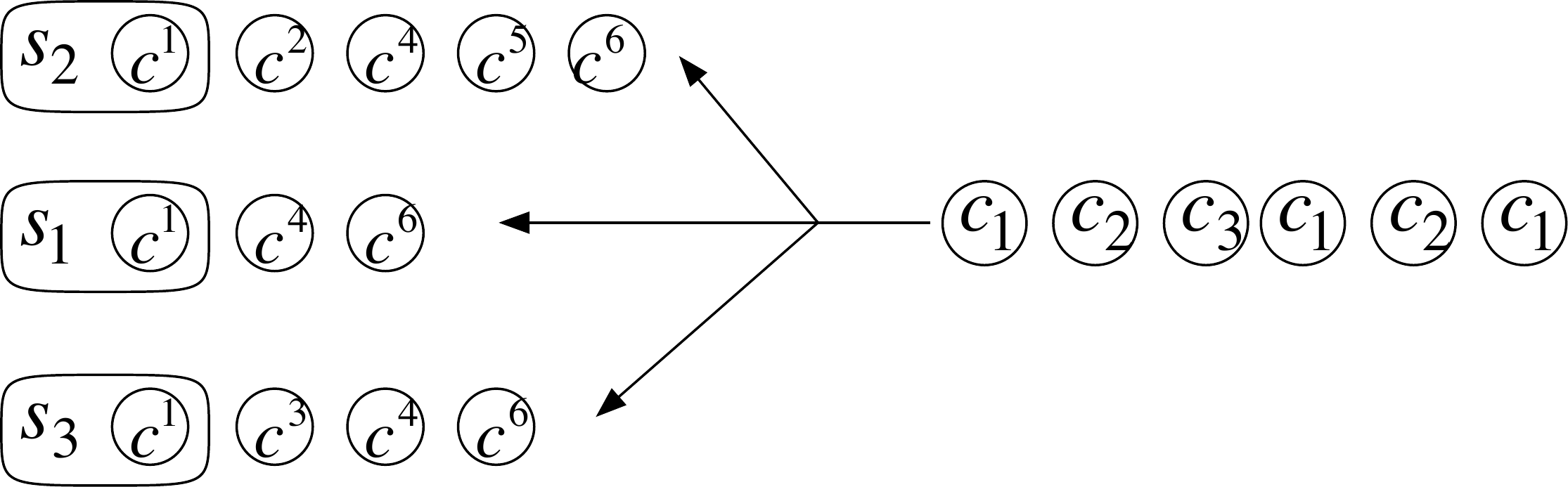} 
   \caption{A current state with redundant queueing}
   \label{fig.redundant}
\end{figure}
In it we display the list of customer types, in order of arrival, on the right side, and on the left side are the servers and their queues.  The first customer, $c^1$ (where the superscript $^1$ indicates his palace in the sequence of customers in the system) is of type $c_1$, and is currently being served simultaneously be all three servers.  The second and third customers are of types $c_2$ and $c_3$ and queue up for servers $s_2,s_3$ respectively.  The fourth and sixth customer, $c^4,c^6$ are again of type $c_1$ and queue up at all  three servers.

Gardner et al. \cite{gardner2016queueing}  have studied this system and defined the following process to describe it:  $X^r(t) = (c^1,\ldots,c^L)$, where $c^1,\ldots,c^L$ are the types of all the customers in the system at time $t$, ordered by their arrival times, with $c^1$ the oldest.  They have shown:
\begin{theorem}[Gardner, Zbarsky, Doroudi, Harchol-Balter, Hyytia and Scheller-Wolf \cite{gardner2016queueing}]
\label{thm.redundancy}
The process $X^r(t)$ is a continuous-time discrete state Markov chain.  It is ergodic if and only if  the stability condition (\ref{eqn.stability}) holds.  Its stationary distribution is given, up to a normalizing constant, by:
\begin{eqnarray}
\label{eqn.redundancy}
&P^r(c^1,c^2,\ldots,c^L) \propto 
\prod_{\ell=1}^L \frac{\lambda_{c^\ell}}{\mu_{\S(\{c^1,\ldots,c^\ell\})}}  
\end{eqnarray}
\end{theorem}

\subsection{The FCFS parallel matching queue}
\label{sec.matchqueue}
Customers of various types arrive in independent  Poisson streams of rates $\lambda_{c_i}$ and queue up in order of arrival.  Servers of various types arrive in independent Poisson streams of rates $\mu_{s_j}$.  An arriving server scans the queue of customers and matches with the longest waiting customer that is compatible with him, and the two leave the system immediately.  If the server does not find a compatible customer in the queue he leaves immediately without a customer.  

Figure \ref{fig.matchqueue}  illustrates a possible history of this system, for our example.  
\begin{figure}[htbp]
   \centering
   \includegraphics[width=1.9in]{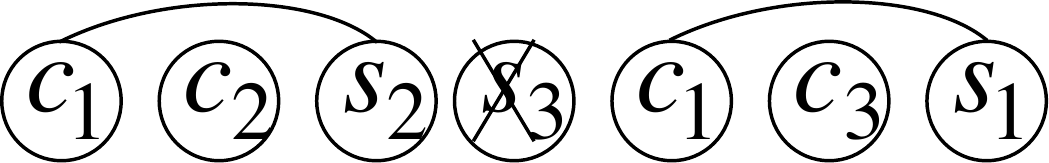} 
   \caption{A partial history of the matching queue}
   \label{fig.matchqueue}
\end{figure}
The figure shows a sequence of customers and servers specified by their types, ordered in the order of arrival from left to right.  Customer of type $c_1$ arrived first, followed by a customer of type $c_2$.  Next a server of type $s_2$ arrived and was immediately matched to the first customer and they departed together.  Next a server of type $s_3$ arrived and left immediately without a match.  This was followed by a customer of type $c_1$, then a customer of type $c_3$ and finally by a server of type $s_1$ that matched immediately with the third customer, of type $c_1$, and departed with him.  At this point in time there was a queue of two customers, the earlier of type $c_2$, the later of type $c_3$.

We describe this system by the process $X^m(t) = (c^1,\ldots,c^L)$, where there are $L$ customers in total, their  types (random) are $c^1,\ldots,c^L$,  ordered in order of arrival, with $c^1$ the longest waiting, and the time is $t$.

\begin{theorem}
\label{thm.matchqueue}
The process $X^m(t)$ is a Markov chain, it is ergodic if and only if the stability condition (\ref{eqn.stability}) holds, and its stationary distribution is given up to a constant by:
\begin{eqnarray}
\label{eqn.matchingqueue}
&P^m(c^1,\ldots,c^L) \propto \prod_{\ell=1}^L \frac{\lambda_{c^\ell}}{\mu_{\S(\{c^1,\ldots,c^\ell\})}}
\end{eqnarray}
\end{theorem}
The proof of this theorem is identical to the proof of Theorem \ref{thm.redundancy}.  It also follows directly from Theorem \ref{thm.equiv1}.

\section{Comparison of the three service models}
\label{sec.relations}
As we see from Theorems \ref{thm.redundancy}, \ref{thm.matchqueue} and Corollary \ref{thm.fcfsalisbusy}, all three parallel service systems are associated with a Markov chain with the same stationary distribution.  Furthermore this stationary distribution is similar to that of the FCFS infinite matching model.  We now explore the relations between  these models.

\subsection{Equivalence of the redundancy service system and the matching queue}
\label{sec.equiv1}
Note that although the redundancy system can have idle servers, and the matching queue cannot, the state of the redundancy system is completely determined by the sequence of customers in the system; servers are idle at a given time if there are no compatible customers  in the system at that time. We will show that the matching and redundancy queues are sample-path equivalent in the sense that if we start them with the same customer state, and we couple the customer arrival processes in the two queues, and we couple potential service completions in the redundancy queue with service arrivals in the matching queue, then the sample paths for the state processes of the two systems will be identical, with probability one.

\begin{theorem}
\label{thm.equiv1}
The redundancy service system and the matching queue are equivalent in the sense that 
the processes $X^r(t)$ and $X^m(t)$ are sample path equivalent.  In particular this means that for each type of customer, the  sojourn time in the system is the same for both models.
\end{theorem}
\begin{proof}
Consider the situation at time $t$ where the customers in the system, ordered in order of arrival, are of types $c^1,\ldots,c^L$, in each of the systems.  Then if a customer of type $c^*$ arrives   
he joins the queue as last in both systems, which remain identical.  The only other thing that can happen is that one of the customers leaves.  In the redundancy service system,  the first customer is currently served by all servers in $\S(c^1)$ simultaneously, and will depart at rate $\mu_{\S(c^1)}$.  In the matching model, the first customer will depart if a server of type in  $\S(c^1)$ arrives, which happens at rate $\mu_{\S(c^1)}$.  Furthermore, 
in the redundancy service system,  the $\ell$th customer is currently served by all servers in $\S(c^\ell)\backslash \S(\{c^1,\ldots,c^{\ell-1}\})$ simultaneously, if this set is non-empty, otherwise it is not being served.  The rate of departure of the $\ell$th customer is then 
$\mu_{\S(c^\ell)\backslash \S(\{c^1,\ldots,c^{\ell-1}\})}$.  In the matching system the $\ell$th customer will depart if a server of type in $\S(c^\ell)\backslash \S(\{c^1,\ldots,c^{\ell-1}\})$ arrives, which has the same rate $\mu_{\S(c^\ell)\backslash \S(\{c^1,\ldots,c^{\ell-1}\})}$.

So, in the coupled system, the first change from state $c^1,\ldots,c^L$ will be the same for both systems.  This completes the proof.
\end{proof}

\subsection{Comparing the FCFS-ALIS and  the redundancy service systems}
\label{sec.equiv2}
In contrast, the situation is different when we compare the FCFS-ALIS system with the redundancy system.   We list some points for comparison:
\begin{compactitem}[-]
\item
The process $X^q(t)|$busy and $X^r(t)$ have the same stationary distribution, but $X^r(t)$ includes all customers in the system, those waiting and those being served, while $X^q(t)$ only includes waiting customers, so there is an additional set of customers which are currently in service in the FCFS-ALIS system.

It is in fact shown that the stationary distribution of the types of customers that are in service in the FCFS-ALIS system cannot be expressed in product form, even for the simple `N' compatibility graph; see \cite{visschers2012product}.
\item
One can regard the FCFS-ALIS system also as a system in which customers split on arrival into several copies that queue up at all the compatible servers, similar to the redundancy queue.  However, at the instant that service of one copy  starts,  all the other copies disappear.  This happens either when the customer has been waiting at several queues, and  reaches the server in one of these queues, or when on arrival he finds several compatible servers, in which case he will be processed by the  longest idle server, so there is no simultaneous processing.  
\item
It is worth mentioning that the FCFS-ALIS system is equivalent to a system in which customers have full information about all the processing times in the system, and each arriving customer  joins the compatible server with the shortest workload.  This  join the shortest workload policy (JSW) leads to  a Nash equilibrium determined by selfish customers.
\item
With the same set of customers  $c^1,\ldots,c^L$, and the same set of idle servers $s^1,\ldots,s^K$ in the system,  
 under FCFS-ALIS each busy server serves a different customer, while in the redundancy system different servers may serve the same customer simultaneously.  Therefore, although the stationary distributions of $X^q(t)|$busy and $X^r(t)$ are the same, they are not sample path equivalent.
\item
Because all the processing times are exponentially distributed, there is no loss of processing time when a customer is served simultaneously by more than one server.  In fact, if a set of servers are processing jobs, the next service completion will be at the same time whether they work on different customers or are processing the same customer simultaneously. 
\item
If in the two systems there is the same set of customers (including both, waiting and in service), then the number of busy servers in the redundancy system is greater or equal to the number of busy servers in the FCFS-ALIS systems, because simultaneous service is allowed under the redundancy system.
\item
Under the Redundancy service policy  flexible customers have an advantage over less flexible customers.  As a result, the composition of  customers in the system under Redundancy service may include more inflexible and fewer flexible customers than  under FCFS-ALIS policy.  This may result in forced idleness when too many inflexible customers accumulate, and there are no  flexible customers left in the system.
\end{compactitem}
The last two considerations indicate that in comparing performance of the two service policies may depend on the parameters of the system, such as workloads and service rates.
  In the next section we take a closer look at this question by  a detailed study of the special case of the `N'-system.


\section{A comparison of FCFS-ALIS and Redundancy Service for the `N'-System}
\label{sec.N-model}
In this section, we compare the performance of  \textit{FCFS-ALIS} policy and  \textit{Redundancy Service} policy for the  `N'-system. In the  comparison of the expected sojourn times and number of customers in steady state under the two policies, we find that neither of  the policies dominates the other. 
We then consider a coupled realization of both systems, and analyze how the sample paths under the two policies differ, and prove a theorem on the difference.  We also present some simulation results that illustrate typical behavior in light traffic and in heavy traffic.

The $N$ model is illustrated in figure \ref{im0}.  There are two servers and two customer types. Type 1 customers arrive at rate $\lambda_1$ and are flexible, and can be served by either server, type 2 customers arrive at rate $\lambda_2$ and can only be served by server 2.  Server 2 is flexible and can serve both types of customers, at rate $\mu_2$, while server 1 can only serve type 1 customers, at rate $\mu_1$.
\begin{figure}[h!]
  \centering
  \includegraphics[scale=0.3]{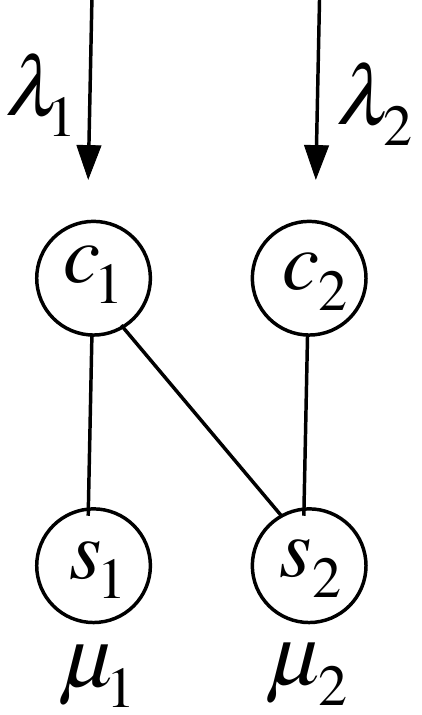}
   \caption{The `N'-system}
  \label{im0}
\end{figure}

   The sojourn time for the `N'-system under the Redundancy Service policy is derived in  Theorems 2 and 3 of \cite{gardner2016queueing}.  From this we obtain the expected sojourn times for type 1 and type 2 customers:
\[
E(W_1^r) = \frac{1}{\mu_1+\mu_2 - \lambda_1 -\lambda_2},  \qquad 
E(W_2^r)= \frac{1}{\mu_2 - \lambda_2} -  \frac{1}{\mu_1+\mu_2 -\lambda_2}  
+ \frac{1}{\mu_1+\mu_2 - \lambda_1 -\lambda_2}. 
\]   
The expected waiting times and the service times for type 1 and type 2 customers under the FCFS-ALIS policy can be calculated using results of section 4 in \cite{visschers2012product,adan2014skill}.  Using these results we obtain first the expected waiting times $V_1^q,\,V_2^q$, and then the expected service times $S_1^q,\,S_2^q$.  The waiting times are:
\[
\begin{array}{l}
E(V_1^q) = B  \frac{1}{(\mu_1+\mu_2 -\lambda_1-\lambda_2)^2}\left( \frac{1}{\mu_1} + \frac{1}{\mu_2 - \lambda_2}  \right), \\ 
\\
 E(V_2^q)=  B\left(\frac{1}{\lambda_1(\mu_2 -\lambda_2)^2} +
 \frac{1}{\mu_1 (\mu_1+\mu_2 -\lambda_1-\lambda_2)^2} +
 \frac{1}{(\mu_2 -\lambda_2)^2(\mu_1+\mu_2 -\lambda_1-\lambda_2)} +
 \frac{1}{(\mu_2 -\lambda_2)(\mu_1+\mu_2 -\lambda_1-\lambda_2)^2} 
  \right)
\end{array}
\]   
where
{\small
\[
B = \left( \frac{1}{\lambda_1(\lambda_1+\lambda_2)} +\frac{1}{(\lambda_1+\lambda_2)^2}
+  \frac{1}{\mu_1 (\lambda_1+\lambda_2)} + 
 \frac{1}{\lambda_1(\mu_2 -\lambda_2)} + 
 \frac{1}{\mu_1 (\mu_1+\mu_2 -\lambda_1-\lambda_2)} +
 \frac{1}{(\mu_2- \lambda_2)(\mu_1+\mu_2 -\lambda_1-\lambda_2)}
\right)^{-1}.
\]
}
To calculate the service times of customers of type 1, we note that the total stationary probability that servers 1 and 2 are busy, denoted here as $b_1^q,\,b_2^q$, are given \cite{adan2014skill} by:
\[
\begin{array}{l}
b_1^q=P(\mbox{server 1 busy})=B \left( \frac{1}{\mu_1}\frac{1}{\lambda_1+\lambda_2} + 
\frac{1}{\mu_1}\frac{1}{\mu_1+\mu_2-\lambda_1-\lambda_2} + 
\frac{1}{\mu_2-\lambda_2}\frac{1}{\mu_1+\mu_2-\lambda_1-\lambda_2}\right),\\
\\
b_2^q=P(\mbox{server 2 busy})=B \left(  \frac{1}{\mu_2-\lambda_2}\frac{1}{\lambda_1} + 
\frac{1}{\mu_1}\frac{1}{\mu_1+\mu_2-\lambda_1-\lambda_2} + 
\frac{1}{\mu_2-\lambda_2}\frac{1}{\mu_1+\mu_2-\lambda_1-\lambda_2}\right).
\end{array}
\] 
and therefore the stationary probability that server 2 is working on customer of type 1 is 
\[
P(\mbox{server 2 working on customers of type 1})  = b_2^q - \frac{\lambda_2}{\mu_2}.
\]
From these we get expressions for the expected service times of the customers in steady state:
\[
E(S_1^q)=\frac{b_1^q}{\mu_1 b_1^q + \mu_2(b_2^q - \frac{\lambda_2}{\mu_2})}
+ \frac{b_2^q - \frac{\lambda_2}{\mu_2}}{\mu_1 b_1^q + \mu_2(b_2^q - \frac{\lambda_2}{\mu_2})},  \qquad  E(S_1^q)= \frac{1}{\mu_2}.
\]

The expected number of customers in the system can now also be obtained, by Little's Law.
This enables us to compare expected sojourn times and number in system under the two policies.  In the following Figure \ref{fig.3dplots} we plot the difference in expected number in system under the two policies.  It is seen from the plots that one policy does not always dominate the other.  
\begin{figure}[h!]
  \centering
  \includegraphics[width=1\textwidth]{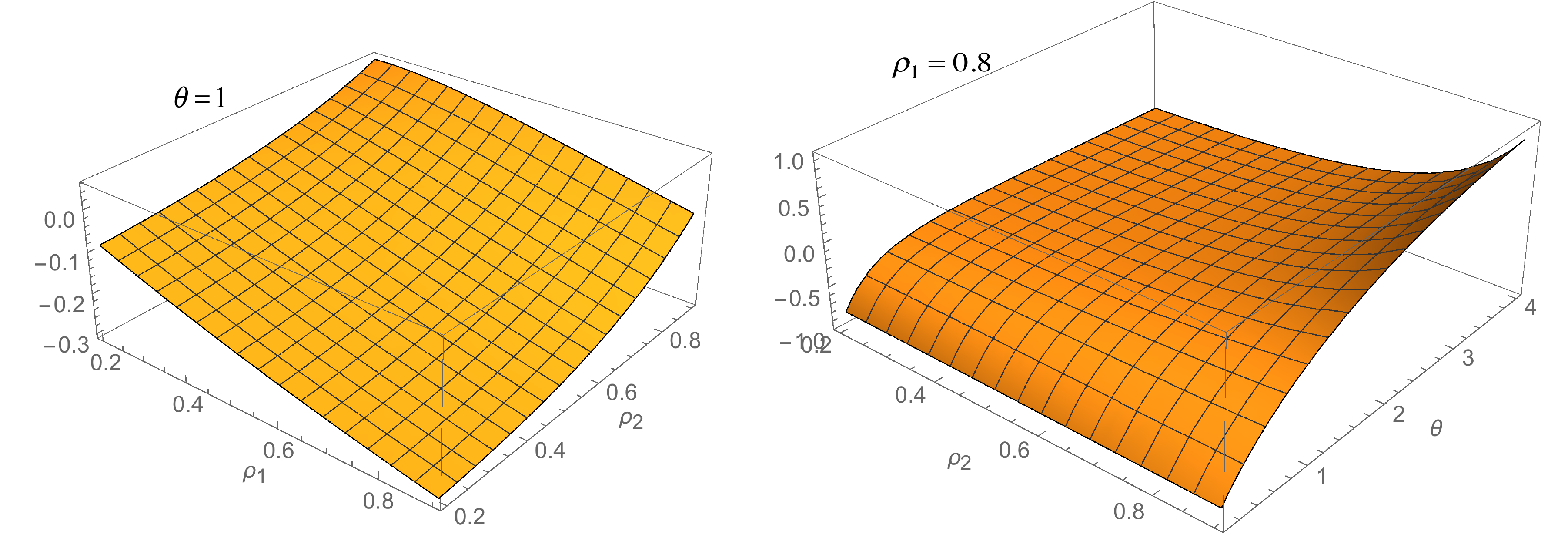}
   \caption{Plots of $E(N_1^r+N_2^r-N_1^q-N_2^q)$ as a function of $\rho_1=\frac{\lambda_1}{\mu_1},\, \rho_2=\frac{\lambda_2}{\mu_2},\,\theta=\frac{\mu_1}{\mu_2}$}
  \label{fig.3dplots}
\end{figure}

To learn more about the behavior of the `N'-system under the two policies, we now study coupled versions of the two systems, which we define as follows:  We have 4 independent Poisson processes of rates $\lambda_1,\lambda_2,\mu_1,\mu_2$ which are shared by the two systems.  The first two give the arrival times of the two types of customers.  The second two give the sequences of potential service completions of the two servers.  When a potential service is completed at either of the systems, if a customer is in service that customer leaves.  In the system with redundant service, if two servers are serving the same customer, then when one of them completes service, the other server will continue the unfinished service period by serving another customer (or it will be idle).  
If the server is idle at a service completion there is no change in the state.  This coupling means that arrival times and potential service completions occur at the same time in both systems, but we may have that the type of customer that leaves is different, or even that under one policy there is a departure, and under the other policy there is no departure.

Let $N_1^q,\, N_1^r,\, N_2^q,\, N_2^r$, denote the random variables counting numbers  of customers in the coupled systems, when the systems are in steady state.  

\begin{theorem}
\label{thm.N-comparison}
The total number in system under Redundant Service is stochastically greater than the total number in the system under FCFS-ALIS minus 1:
\begin{equation}
\label{eqn.comparison}
N_1^r + N_2^r \ge_{ST}  N_1^q + N_2^q -1. 
\end{equation}
\end{theorem}
The proof of this theorem is based on coupling, and is given in \ref{sec.N-comparison}.  

It is instructive to explain under what conditions either policy may be advantageous.
In light traffic, when the servers are not overloaded, the redundant system may be preferred, 
since there will be many occasions when there is a single type 1 customer in the system, in which case under FCFS-ALIS one of the servers is idle, while under Redundancy Service both servers are working.  

This is illustrated in the following coupled  sample realization in Figure \ref{figst3}.  The parameters for this example are:
\begin{figure}[h!]
  \centering
  \includegraphics[width=1\textwidth]{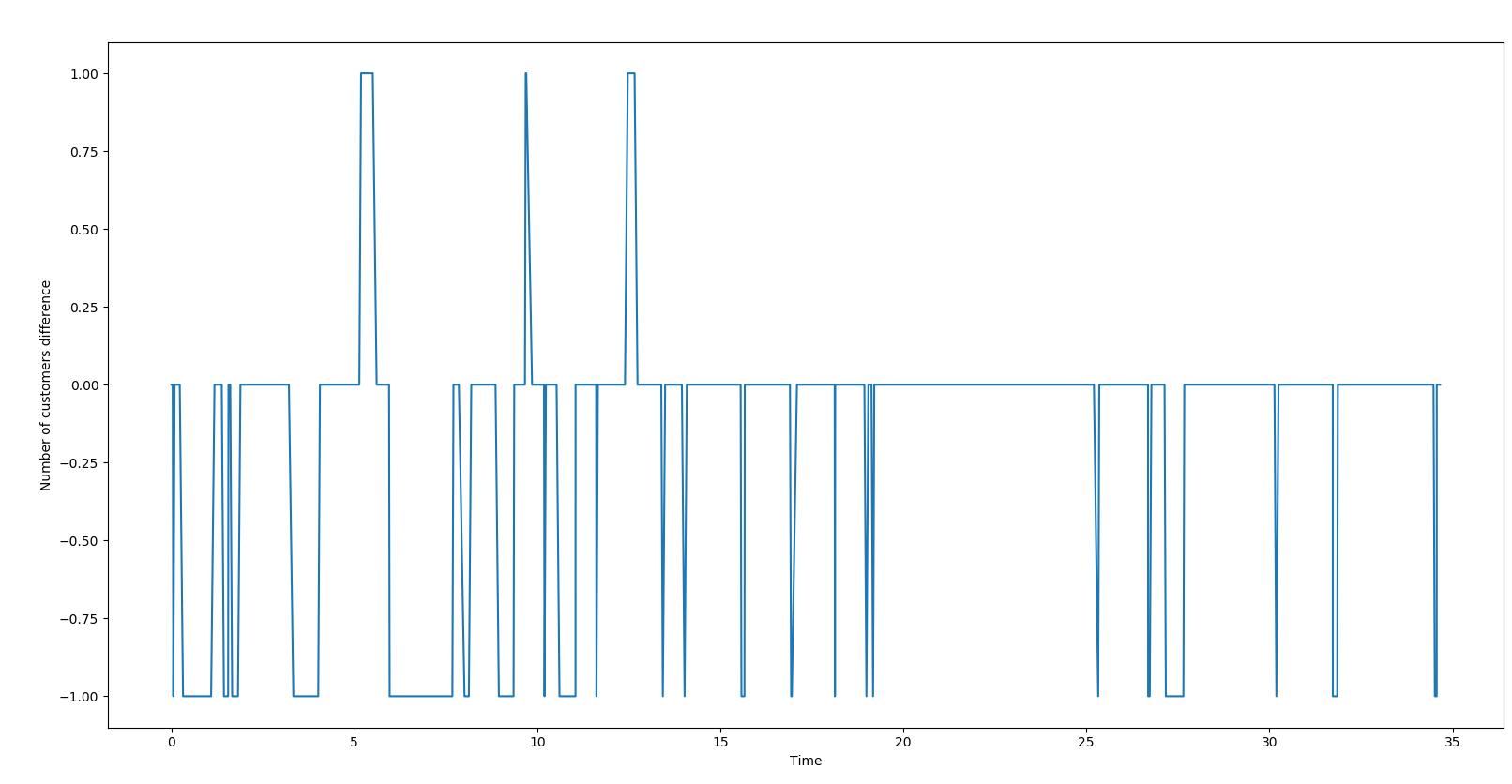}
   \caption{Example 1:  Difference in the number of customers  for: $\lambda_1=3,\lambda_2=2,\mu_1=6,\mu_2=6$}
  \label{figst3}
\end{figure}
\[
\mbox{Example 1:} \quad \lambda_1=3,\,\lambda_2=2,\,\mu_1=6,\,\mu_2=6,
\qquad E(N_1^r+N_2^r)=1.014, \quad E(N_1^q+N_2^q)=1.194.
\]
We plot the number of customers $N_1^r+N_2^r - N_1^q-N_2^q$.  We see that for an appreciable fraction  of the time the difference  equals to $-1$.

On the other hand, when the flexible server 2 is heavily loaded by inflexible customers  of type 2, then at each time that server 2 is `helping' server 1 by serving a type 1 cusotmer, customers of type 2 accumulate, and so redundancy can have more congestion than FCFS-ALIS policy.
This is clearly illustrated in the following coupled  sample realization in Figure \ref{figst2}.
\begin{figure}[h!]
  \centering
  \includegraphics[width=1\textwidth]{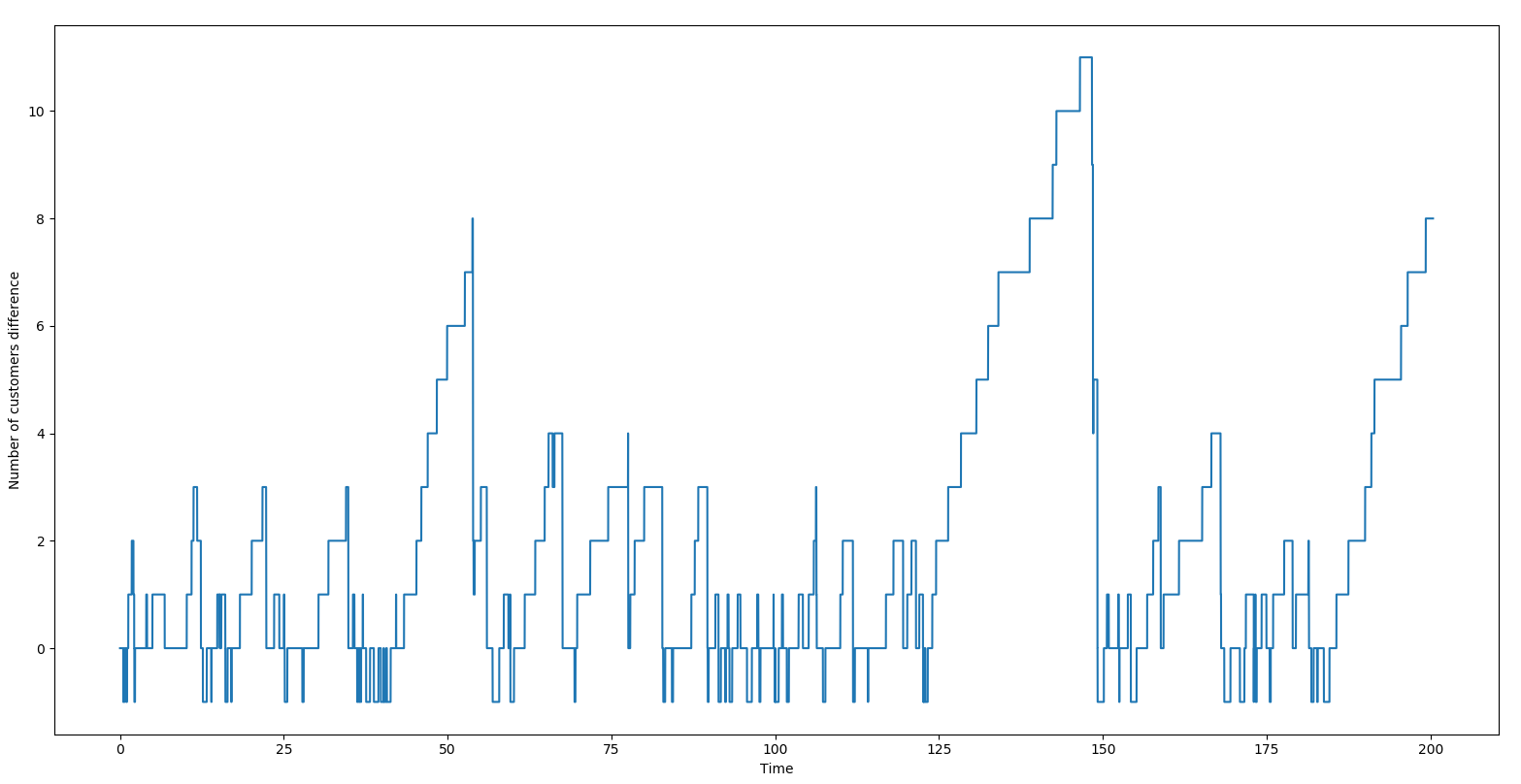}
   \caption{Example 2: Difference in the number of customers for: $\lambda_1=2,\lambda_2=145,\mu_1=3,\mu_2=150$}
  \label{figst2}
\end{figure}
\[
\mbox{Example 2:} \quad \lambda_1=2,\,\lambda_2=145,\,\mu_1=3,\,\mu_2=150,
\qquad E(N_1^r+N_2^r)=35.375, \quad E(N_1^q+N_2^q)=32.4993.
\]
  We see that the difference is positive most of the time and can be as high as $10$.  Note that in this case, if we use dedicated service of customers of type 1 by server 1 only, and customers of type 2 by server 2 only, then $E(N_1)=2$, $E(N_2)=29$.


\section{The FCFS Infinite Bipartite Matching Model}
\label{sec.infmatching}
We note that all three service models have similar stationary distributions.  Furthermore, these distributions are also similar to stationary distributions that were obtained for the infinite bipartite matching model introduced by Kaplan, Caldentey and Weiss \cite{caldentey2009fcfs} and further studied in \cite{adan2012exact,adan2015reversibility}.  We briefly describe the results of \cite{adan2012exact,adan2015reversibility}.

There are two independent doubly infinite series of customers $\ldots,c^{-2},c^{-1},c^0,c^1,c^2,\ldots$ drawn i.i.d. from $\C$ according to the probabilities $\alpha$, and of servers $\ldots,s^{-2},s^{-1},s^0,s^1,s^2,\ldots$, drawn i.i.d. from $\S$ according to the probabilities $\beta$,  
and they are matched FCFS according to the compatibility graph $\mathcal{G}$.  

What we mean by FCFS is that if $s^n$ is matched with $c^m$, then there is no earlier $s^k\in \S(c^m)$ which is unmatched, and no earlier $c^\ell \in \C(s^n)$ which is unmatched.
Figure \ref{fig.infmatch}  illustrates FCFS infinite bipartite matching with the compatibility graph of Figure \ref{fig.compatiiblity}, for a window of the sequences.  
\begin{figure}[htbp]
   \centering
   \includegraphics[width=2.9in]{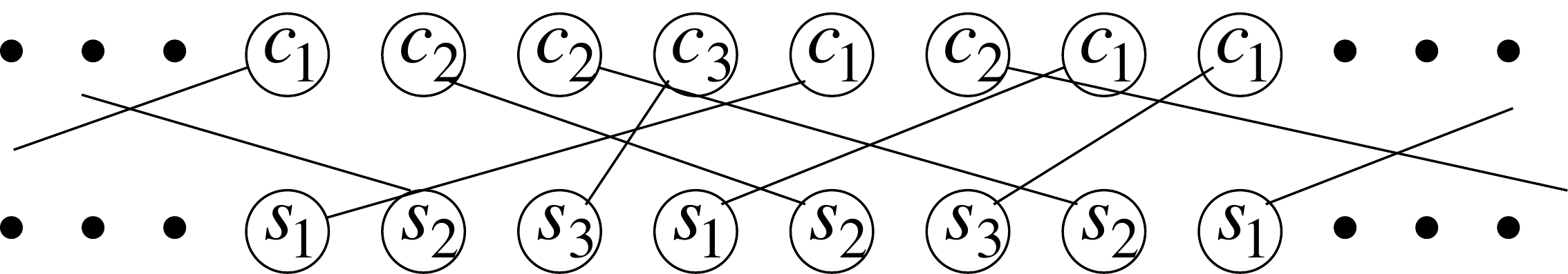} 
   \caption{FCFS infinite  bipartite matching}
   \label{fig.infmatch}
\end{figure}
In this figure one customer and one server were matched to an earlier (left of the window) customer and server, and one customer and one server remain unmatched at the end of the window, and are matched to a later (right of the window) customer and server.
\begin{definition}
We say that this system has complete resource pooling if the following equivalent conditions hold for any  $S\subset \S,\,S\ne\emptyset,\S$ and $C\subset \C,\,C\ne\emptyset,\C$:
\begin{equation}
\label{eqn.comprespool}
\alpha_C < \beta_{\S(C)}, \qquad \beta_S < \alpha_{\C(S)}, \qquad \alpha_{\U(S)} < \beta_S.
\end{equation}
\end{definition}
The following theorem was proved by Adan et al. \cite{adan2015reversibility}:
\begin{theorem}[Adan, Busic, Mairesse and Weiss \cite{adan2015reversibility}]
\label{thm.infmatch1}
If complete resource pooling (\ref{eqn.comprespool}) holds then almost surely there exists a  FCFS matching of the two sequences and this matching is unique.
\end{theorem}  

We define the following transformation on the matched sequences:
\begin{definition}
For  given matched sequences, the exchange transformation exchanges the position of each matched pair, so that if $s^n$ was matched to $c^m$ in the original system, then in the exchanged system we have $\tilde{c}^n$ matched to $\tilde{s}^m$.  This defines a permutation of the original sequence $\ldots,c^{-2},c^{-1},c^0,c^1,c^2,\ldots$ to a new sequence 
$\ldots,\tilde{c}^{-2},\tilde{c}^{-1},\tilde{c}^0,\tilde{c}^1,\tilde{c}^2,\ldots$, and the original sequence $\ldots,s^{-2},s^{-1},s^0,s^1,s^2,\ldots$ to a new sequence 
$\ldots,\tilde{s}^{-2},\tilde{s}^{-1},\tilde{s}^0,\tilde{s}^1,\tilde{s}^2,\ldots$. 
\end{definition}
Figure \ref{fig.exchanged}  illustrates the exchanged sequences obtained by the exchange transformation from the illustration in Figure  \ref{fig.infmatch}.
\begin{figure}[htbp]
   \centering
   \includegraphics[width=2.9in]{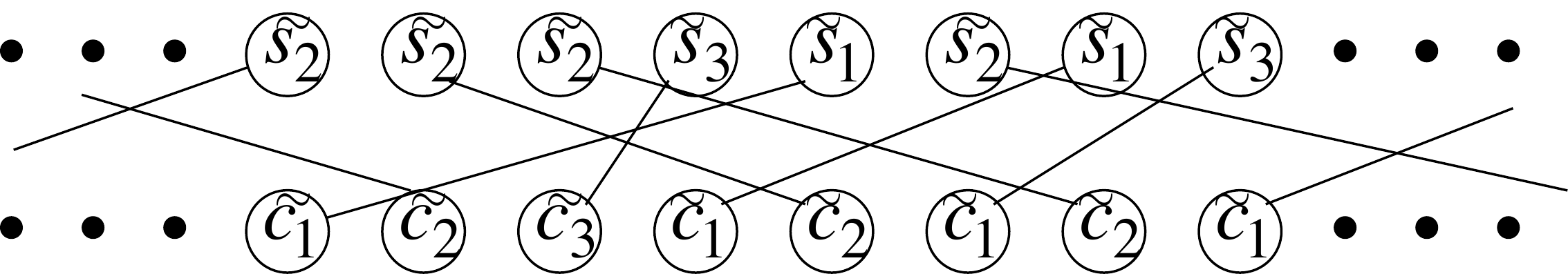} 
   \caption{The exchange transformation applied to Figure \ref{fig.infmatch}}
   \label{fig.exchanged}
\end{figure}

The following reversibility result is proved in \cite{adan2015reversibility}
\begin{theorem}[Adan, Busic, Mairesse and Weiss \cite{adan2015reversibility}]
\label{thm.infmatch2}
The exchanged sequences $\ldots,\tilde{c}^{-2},\tilde{c}^{-1},\tilde{c}^0,\tilde{c}^1,\tilde{c}^2,\ldots$, $\ldots,\tilde{s}^{-2},\tilde{s}^{-1},\tilde{s}^0,\tilde{s}^1,\tilde{s}^2,\ldots$, are independent and each is i.i.d. from $\C$ and from $\S$ according to $\alpha,\,\beta$.  Furthermore, the original matching is now the almost surely unique FCFS matching of the exchanged sequences in reversed time.
\end{theorem}  

Using the reversibility, it is easy to obtain stationary distributions for several Markov chains associated with  this system.  We consider making all the FCFS matches of $s^\ell,c^k$ for $k,\ell \le n$, and define the process  $X^\infty(n)= (c^{i_1},\ldots,c^{i_L},s^{j_1},\ldots,s^{j_L})$,  
where customers in positions $i_1,\ldots,i_L$ and servers in positions $j_1,\ldots,j_L$ 
were left unmatched, and $c^{i_1},\ldots,c^{i_L},s^{j_1},\ldots,s^{j_L}$ are the types of these unmatched customers and servers.  
\begin{theorem}[Adan, Busic, Mairesse and Weiss \cite{adan2015reversibility}]
\label{thm.infmatch3}
The process $X^\infty(n)$ is a discrete time discrete state Markov chain.  It is ergodic if and only if complete resource pooling (\ref{eqn.comprespool}) holds.  Its stationary distribution is given, up to a normalizing constant, by:
\begin{eqnarray}
\nonumber
& P^\infty(c^{i_1},c^{i_2},\ldots,c^{i_L},s^{j_1},\ldots,s^{j_L}) \propto 
\prod_{\ell=1}^L \frac{\alpha_{c^{i_\ell}}}{\beta_{\S(\{c^{i_1},\ldots,c^{i_\ell}\})}}  \\
\label{eqn.infmatch} 
&\times \prod_{\ell=1}^L \frac{\beta_{s^{j_\ell}}}{\alpha_{\C(\{s^{j_1},\ldots,s^{j_\ell}\})}} 
\end{eqnarray}
\end{theorem}  

We note the close resemblance of this formula to the stationary distributions derived in (\ref{eqn.fcfsalis}), (\ref{eqn.redundancy}), (\ref{eqn.matchingqueue}).

\section{A Novel  FCFS Infinite Directed Matching Model}
\label{sec.embedded}
The similarity of the stationary distributions of the processes $X^q$, $X^r$, $X^m$ and the FCFS infinite bipartite matching process $X^\infty$ suggests that they may be more closely related.  In this section we introduce a new FCFS infinite matching model.  It is similar to the model of Section \ref{sec.infmatching} and  \cite{caldentey2009fcfs,adan2012exact,adan2015reversibility}.  It is also related to the model studied in \cite{mairesse2016stability}, and is even more closely related to $X^q$, $X^r$, $X^m$.
We use this new process to derive some more properties of $X^q$, $X^r$, $X^m$, in Section \ref{sec.surprise}.

We consider a single infinite sequence of customers and servers, which is generated as follows:  each successive item in the list is a customer of type $c_i$ with probability $\alpha_{c_i}=\frac{\lambda_{c_i}}{\rlambda + \rmu}$, and it is a server of type $s_j$ with probability $\beta_{s_j}=\frac{\mu_{s_j}}{\rlambda + \rmu}$, and successive items in the sequence are independent.   The result is a sequence $\ldots,z^1,z^2,\ldots$, where each item $z^n$  indicates either a type of customer or a type of server.  We then perform FCFS matching of the customers and servers according to the compatibility graph $\mathcal{G}$, {\em utilizing only matches of servers to earlier customers}.  
This means in particular that a server $z^n=s_j$ for which there is no earlier unmatched compatible customer will remain unmatched.  
We call this the {\em FCFS single stream infinite directed  bipartite matching model},  directed matching model for short.

We define the process $X^{\downarrow\infty}(n)$ to describe the matching process for the directed matching model.   Assume we have performed all the possible matches in the sequence $\ldots,z^1,z^2,\ldots$ up to and including $z^n$.  Then $X^{\downarrow\infty}(n)= (c^1,\ldots,c^L)$ is the ordered list of the  customers that are still unmatched.  
\begin{theorem}
\label{thm.embedded}
$X^{\downarrow\infty}(n)$ is a discrete time discrete state Markov chain, it is ergodic if and only if the stability condition (\ref{eqn.stability}) holds, and its stationary distribution, up to a normalizing constant, is given by:
\begin{equation}
\label{eqn.embedded}
P^{\downarrow\infty}(c^1,\ldots,c^L) \propto \prod_{\ell=1}^L \frac{\lambda_{c^\ell}}{\mu_{\S(\{c^1,\ldots,c^\ell\})}}.
\end{equation}
The fraction of servers that remain unmatched is $1 - \frac{\rlambda}{\rmu}$.
\end{theorem}
\begin{proof}
It is seen immediately that the Markov chain $X^{\downarrow\infty}(n)$ is the jump chain of the process $X^m(t)$.  Furthermore,  the process $X^m(t)$  has jumps in which its state changes at the uniform times of a Poisson process of rate $\rlambda+\rmu$.  The theorem  follows.
\end{proof}

{\bf Remark:}  The fraction of unmatched servers of each type $s_j$ can be calculated from (\ref{eqn.embedded}).  It is the sum of $P^{\downarrow\infty}(c^1,\ldots,c^L)$ over all sequences $c^1,\ldots,c^L$  that contain only customers that are incompatible with $s_j$.

We  define a more detailed process to describe the dynamics of the FCFS  infinite directed matching model. The  process $U(n)=(u^1,\ldots,u^K)$ records the ordered sequence of the unmatched customers as well as the  servers that are left unmatched between them, after all matches of customers and servers in the sequence $\ldots,z^1,z^2,\ldots$ up to and including  $z^n$ have been made.   We refer to $U(n)$ as the augmented Markov chain of the infinite directed matching process.  Here  $U(n)$ starts with  the earliest customer that remained unmatched up to $z^n$, $u^1 \in \C$.  If all customers up to $z^n$ have been matched we say that the matching is perfect, and we define $U(n)=\emptyset$ (we also denote it by $0$).  Clearly by Theorem \ref{thm.embedded}, if the stability condition (\ref{eqn.stability}) holds, then $U(n)$ is an ergodic Markov chain.

We now formulate two theorems for  the FCFS infinite directed matching model.  Their proofs are similar to the proof of Theorems \ref{thm.infmatch1}, \ref{thm.infmatch2} of Section \ref{sec.infmatching}, (they are Theorems 3 and 4 in Adan et al. \cite{adan2015reversibility}).  We include the proofs in  \ref{sec.uniquenss}, \ref{sec.reversal}.
\begin{theorem}
\label{thm.infbackmatch1}
Let $\ldots,z^{-1},z^0,z^1,\ldots$ be a sequence of customer and server types defined as above.  If (\ref{eqn.stability}) holds then almost surely there  exists a directed FCFS matching of servers to cover all the customers, and this matching is unique.
\end{theorem}

We define an exchange transformation for the FCFS infinite directed  matching model:
\begin{definition}
For the FCFS infinite directed  matching model, if all matches were made on $\ldots,z^{-1},z^0,z^1,\ldots$, we define the exchanged sequence $\ldots,\tz^{-1},\tz^0,\tz^1,\ldots$ as follows:  If $z^m=c_i$ was matched to $z^n=s_j$, where $m<n$, then in the exchanged sequence we will have $\tz^m=s_j$, $\tz^n=c_i$.  If $z^n=s_j$ was unmatched, then $\tz^n=z^n$.
\end{definition}

\begin{figure}[htbp]
   \centering
   \includegraphics[width=2.9in]{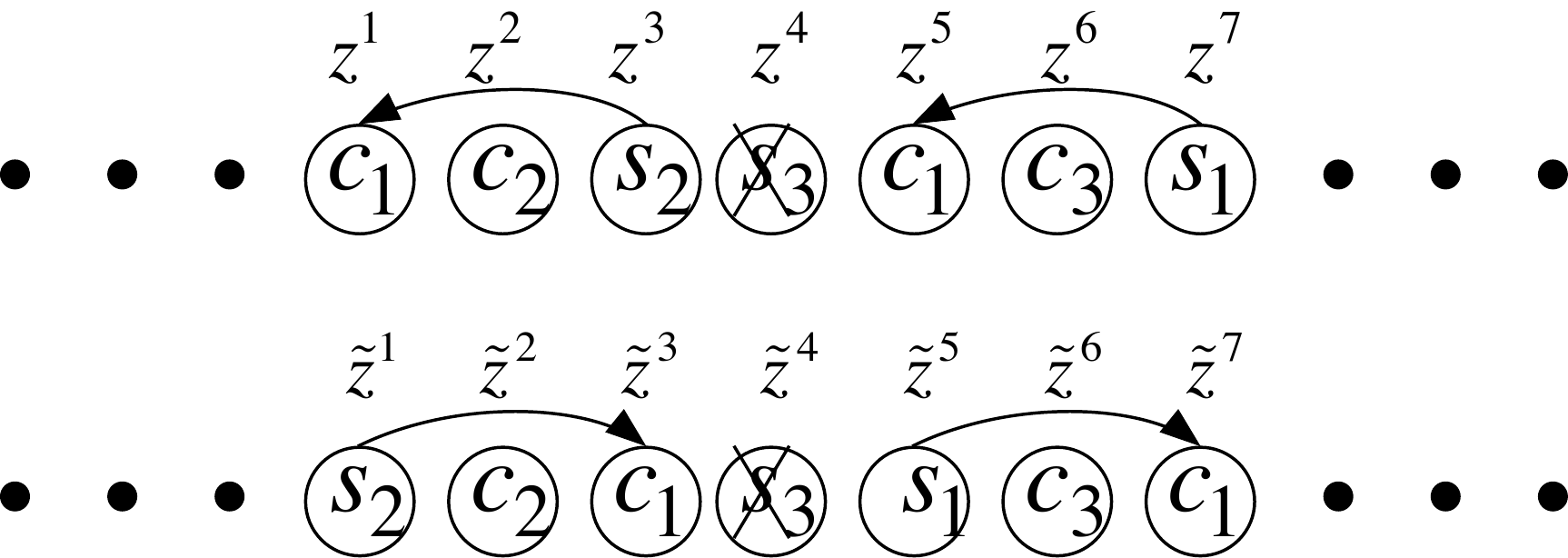} 
   \caption{directed matching and its reversal}
   \label{fig.interleaved}
\end{figure}

Figure \ref{fig.interleaved} describes directed matching for a window of time  in the doubly infinite  sequence of customers and servers on the top panel.  In it, $z^3=s_2$ is matched with earlier $z^1=c_1$, and $z^7$ is matched with $z^5$, while $z^2=c_2,\,z^6=c_3$ are not yet matched, and $z^4=s_3$ will remain unmatched for ever.  
On the bottom panel of  Figure \ref{fig.interleaved} we see the exchange transformation of the top panel, with the matchings directed in reversed time. 

\begin{theorem}
\label{thm.infbackmatch2}
The sequence  $\ldots,\tz^{-1},\tz^0,\tz^1,\ldots$ obtained from the sequence 
$\ldots,z^{-1},z^0,z^1,\ldots$ by the exchange transformation is an i.i.d. sequence.  The unique directed matches for the new sequence performed in reversed time, result in exactly the reversed matches of the original sequence, almost surely.
\end{theorem}


\section{Embeddings and a Version of Burke's Theorem}
\label{sec.surprise}
We have already noticed, and used it in the proof of Theorem \ref{thm.embedded}, that the process $X^{\downarrow\infty}(n)$ is the jump process of the continuous-time matching queue process $X^m(t)$.  
Likewise, by Theorem \ref{thm.equiv1} it is the jump process of $X^r(t)$.

We now use this embedding to prove a version of Burke's Theorem for the FCFS parallel service system under the redundancy policy, and for the FCFS parallel servers matching queue.

\begin{theorem}
\label{thm.inputoutput}
Let $D_{c_i}(t),\, i=1,\ldots,I$ be the departure process of customers of type $c_i$ from the  stationary parallel FCFS  Redundancy Service queue, or from the stationary parallel FCFS Matching queue.  

(i) $D_{c_i}(t)$ are independent Poisson processes of rates $\lambda_{c_i}$.

(ii)  The ordered sequence of customers in the system at time $t$ for either system, given by $X^m(t)$ or $X^r(t)$, is independent of past departures,  $D_{c_i}(s),\, i=1,\ldots,I$ for all  $s<t$.
\end{theorem}
\begin{proof}
We will use the reversibility result of $X^{\downarrow\infty}(n)$ in Theorem \ref{thm.infbackmatch2}.

We consider a  path of $X^m(t),\,-\infty<t<\infty$ (the same goes for $X^r(t)$).  The sample path is determined  by the doubly infinite sequence of arriving customers and servers, $\ldots,z^{-1},z^0,z^1,\ldots$, and the path of the independent   Poisson process of rate $\rlambda+\rmu$, which determines the arrival time of $z^n$ at $t_n$.  The sequence then determines a sample path of the FCFS infinite directed bipartite matching process $X^{\downarrow\infty}(n)$, with the relation that $X^m(t) =  X^{\downarrow\infty}(n)$ in the interval $[t_n,t_{n+1})$.  Consider now the FCFS infinite directed matching for the sequence 
$\ldots,z^{-1},z^0,z^1,\ldots$.   In the matching process, if $z^m=c_i$ is matched to $z^n=s_j$ where $m<n$, then a customer of type $c_i$  arrived at time $t_m$, and a server $s_j$  arrived at $t_n$ and was matched to that customer, and so the customer of type $c_i$ that arrived at time $t_m$ departed at time $t_n$.  

Now we perform the exchange transformation, so we now have the exchanged sequence $\ldots,\tz^{-1},\tz^0,\tz^1,\ldots$.  We then proceed with FCFS directed matching for the exchanged sequence, in reverse order.  By Theorem \ref{thm.infbackmatch2}  in this FCFS directed matching in reverse order, the same pairs will be matched (almost surely), so now $\tz^n=c_i$ is matched with $\tz^m=s_j$.   

Consider now the process $X^m(t)$, and its jump process $X^{\downarrow\infty}(n)$.  Take   the  exchange transformation of the sequence of FCFS directed matchings,  reverse the discrete time, and perform FCFS directed matching,  to obtain the discrete time exchanged and reversed process $\overleftarrow{X}^{\downarrow\infty}(n)$, and using the reversed sequence of time intervals between jumps in $X^m(t)$, construct from  $\overleftarrow{X}^{\downarrow\infty}(n)$ the continuous-time process  $\overleftarrow{X}^{m}(t)$.

By Theorem  \ref{thm.infbackmatch2}, the stationary $X^{\downarrow\infty}(n)$ and the stationary $\overleftarrow{X}^{\downarrow\infty}(n)$ are stochastically identical.  The Poisson process of arrival with rates $\rlambda+\rmu$ is time reversible and so  $X^m(t)$ and $\overleftarrow{X}^{m}(t)$ are stochastically identical.    in particular, the sequence of arrivals of $\overleftarrow{X}^{m}(t)$ consists of independent Poisson process of arrivals of customers of type $c_i$ at rates $\lambda_{c_i}$, and the state of the process, $X^m(t)$ is independent of the arrivals at all time $s>t$.  But these arrivals are exactly the departures of $X^m(t)$ in reversed time.  This completes the proof.
\end{proof}

\begin{corollary}
Networks of parallel service systems under the redundancy service policy, as well as networks of parallel matching queues have product form stationary distributions.
\end{corollary}
\begin{proof}
  This version of Burke's Theorem, as given by Theorem \ref{thm.inputoutput}, implies that the process   $X^{\downarrow\infty}(n)$ is quasi-reversible.   It is proven in \cite{kelly2011reversibility,walrand1988introduction} that networks of quasi reversible Markovian systems have a product form stationary distribution. 
\end{proof}

Another consequence of the embedding is a relaxation of the Poisson-exponential assumptions.

\begin{theorem}
\label{thm.relax}
The stationary distribution of the FCFS matching queue $X^m(t)$, at times $t$ immediately following transitions, remains the same as given in (\ref{eqn.matchingqueue}) if the arrivals are a general stationary point process,  as long as types of arrivals are i.i.d. so that each arrival is a customer of type $c_i$ with probability $\frac{\lambda_{c_i}}{\rlambda + \rmu}$, and it is a server of type $s_j$ with probability $\frac{\mu_{s_j}}{\rlambda + \rmu}$.
\end{theorem}

\begin{proof}
Consider the FCFS matching queue model, when arrivals are a stationary point process, 
and the arrival types are i.i.d. as above.  Let $t_n,\,n = \ldots,-1,0,1,\ldots$ denote the arrival times.  Then at the times $t_n$ the ordered sequence of customers in the system, given by $X^m(t_2)$  is exactly the state of $X^{\downarrow\infty}(n)$, where in this FCFS infinite directed matching model arrivals are i.i.d., and so the stationary distribution of $X^{\downarrow\infty}(n)$ is given by (\ref{eqn.embedded}), which is the same as (\ref{eqn.matchingqueue}).  So the stationary distribution of $X^m(t)$ at the times $t_n$ is given by (\ref{eqn.matchingqueue}). 
\end{proof}

\subsection{Another Interpretation of the Matching Queue}
In our description and interpretation of the parallel matching queue we thought of customers waiting for servers, and arriving servers match to the oldest waiting compatible customer, or are lost.  This system is stable when $\rmu > \rlambda$.  It may describe a situation in which patients are waiting for a transplant of an organ, and patients have enough patience to wait for the right organ to arrive, and the supply of organs is sufficient, but organs cannot be conserved.

In reality the situation may be different, the organs may be conserved for a while, but there are more patients than organs.  So now $\rlambda > \rmu$, and patients may be lost.  We now consider the following process and policy:  servers arrive and queue up waiting for customers, customers arrive, and each arriving customer then matches to the longest waiting compatible server and leaves immediately, or if no compatible server is found, the customer leaves immediately without a match.  All we did in this model is to switch the roles of customers and servers, and all the results of Sections \ref{sec.matchqueue} and \ref{sec.surprise}  hold, with $c_i$ and $s_j$ switching roles.  Denote by $Y^m(t)=(s^1,\ldots,s^L)$ the process that records the ordered sequence of available servers at time $t$, with $s^1$ the longest waiting.
Then the stationary distribution of  $Y^m(t)$ is given by:
\begin{equation}
\label{eqn.transplant}
P(Y^m(t) = s^1,\ldots,s^L) \propto \prod_{\ell=1}^L \frac{\mu_{s^\ell}}{\lambda_{\C(\{s^1,\ldots,s^\ell\})}}.
\end{equation}

\subsection{Embedding the FCFS-ALIS queues in an Infinite Matching Model}
The process $X^q(t)$ can also be embedded in an infinite matching model, by considering the same sequences $\ldots,z^{-1},z^0,z^1,\ldots$, but using a different matching mechanism:  We now  match each successive server $z^n=s_j$ to the earliest unmatched compatible customer $z^m=c_i$ where $m < k$ and $k$ is the earliest position in the sequence with $k>n,\,z^k=s_j$.  If no such match exists, the server $z^n$ remains unmatched.   

We define the process $X^{q\infty}(n)$ to describe the system after all possible matches that involve server $z^k$ and customer $z^\ell$ for all $k,\ell \le n$ have been made.    
Then $X^{q\infty}(n)=(c^1,c^2,\ldots,c^L,s^1,\ldots,s^K)$.  Here $c^1,c^2,\ldots,c^L$ are the types of the customers in positions $\le n$ that are still unmatched, ordered as they appeared in the sequence, and $s^1,\ldots,s^K$ are the types of servers in positions $\le n$ that have not been matched but may still be matched to a customer later in the sequence, ordered as they appeared in the sequence.  Note that any of $c^1,c^2,\ldots,c^L$ are incompatible with any of $s^1,\ldots,s^K$, and that the server types $s^1,\ldots,s^K$ are all different, so that $K\le J$.
   
One can  see that this process is the discrete time jump process of $X^{q}(t)$, and analogues of Theorems \ref{thm.embedded} and \ref{thm.infbackmatch1} hold.

\appendix
\section*{Appendix: Completion of Proofs}
\section{FCFS-ALIS stationary distribution}
\label{sec.partialbalance}
\begin{proof}[Proof of Theorem \ref{thm.fcfsalis}]
The proof is by verifying that (\ref{eqn.fcfsalis}) satisfies partial balance.  It is similar to the proof of Theorem \ref{thm.adanweiss} given in \cite{visschers2012product,adan2014skill}, and to the proof of Theorem \ref{thm.redundancy} given in \cite{gardner2016queueing}. 

We consider a state  $x = (c^1,\ldots,c^L,s^1,\ldots,s^K)$.  
We list transitions in and out of the state $x$ and their rates:
\begin{compactenum}[(i)]
\item   
Transition out of $x$ due to arrival of type $c_i$,  that joins the queue,  rate $\lambda_{c_i}$,
where $c_i \not\in \C(\{s^1,\ldots,s^K\})$
\item
Transition out of $x$ due to arrival of type $c_i$,   that matches to one of the idle servers, at rate 
$\lambda_{\C(\{s^1,\ldots,s^K\})}$.  
\item
Transition out of $x$ due to completion of service, where server type $s_j$ becomes idle, at rate:
$\mu_{s_j}$, for $s_j \not\in S(\{c^1,\ldots,c^L\})$.
\item
Transition out of $x$ due to completion of service and start of service of a waiting customer, at rate:  $\mu_{\S(\{c^1,\ldots,c^L\})}$
\item 
Transition into state $x$ due to arrival of $c^L$, at rate $\lambda_{c^L}$.
\item
Transition into state $x$ due to an arrival that matched with idle server $s^*$ that was in position $k+1$, at rate:
$\lambda_{\C(s^*)\backslash \C(\{s^1,\ldots,s^k\})}$, where $s^* \not\in \S(\{c^1,\ldots,c^L\})$
\item
Transition into state $x$ due to a service completion, and server becoming idle, at rate $\mu_{s^K}$.
\item
Transition into state $x$ due to a service completion, where a server is starting service of a customer $c^*$ that was in position $\ell+1$, at rate:  $\mu_{\S(c^*)\backslash \S(\{c^1,\ldots,c^\ell\})}$.
\end{compactenum}

We now show by substitution of the conjectured values from (\ref{eqn.fcfsalis}), that partial balance equations hold.

{\sc $\bullet$ Balance of (iv) with (v):}
\[
\begin{array}{l} \displaystyle
 P^q(c^1,\ldots,c^L,s^1,\ldots,s^K) \times \mu_{\S(\{c^1,\ldots,c^L\})} =  \\\displaystyle
\prod_{\ell=1}^L \frac{\lambda_{c^\ell}}{\mu_{\S(\{c^1,\ldots,c^\ell\})}} 
\prod_{k=1}^K \frac{\mu_{s^k}}{\lambda_{\C(\{s^1,\ldots,s^k\})}} \times
\mu_{\S(\{c^1,\ldots,c^L\})}; \\   
\\ \displaystyle 
 P^q(c^1,\ldots,c^{L-1},s^1,\ldots,s^K) \times \lambda_{c^L} =  \\ \displaystyle
\prod_{\ell=1}^{L-1} \frac{\lambda_{c^\ell}}{\mu_{\S(\{c^1,\ldots,c^\ell\})}} 
\prod_{k=1}^K \frac{\mu_{s^k}}{\lambda_{\C(\{s^1,\ldots,s^k\})}}
\times \lambda_{c^L}.
\end{array}
\]

{\sc $\bullet$ Balance of (ii) with (vii):}
\[
\begin{array}{l} \displaystyle
 P^q(c^1,\ldots,c^L,s^1,\ldots,s^K) \times \lambda_{\C(\{s^1,\ldots,s^K\})} =  \\\displaystyle
\prod_{\ell=1}^L \frac{\lambda_{c^\ell}}{\mu_{\S(\{c^1,\ldots,c^\ell\})}} 
\prod_{k=1}^K \frac{\mu_{s^k}}{\lambda_{\C(\{s^1,\ldots,s^k\})}} \times
\lambda_{\C(\{s^1,\ldots,s^K\})}; \\  
\\ \displaystyle 
 P^q(c^1,\ldots,c^L,s_1,\ldots,s^{K-1}) \times \mu_{s^K} =  \\ \displaystyle
\prod_{\ell=1}^L \frac{\lambda_{c^\ell}}{\mu_{\S(\{c^1,\ldots,c^\ell\})}} 
\prod_{k=1}^{K-1} \frac{\mu_{s^k}}{\lambda_{\C(\{s^1,\ldots,s^k\})}}
\times \mu_{s^K}.
\end{array}
\]

{\sc $\bullet$ Balance of (i) with (viii):}

For $c_i \not\in \C(\{s^1,\ldots,s^K\})$
\[
\begin{array}{l} \displaystyle
 P^q(c^1,\ldots,c^L,s^1,\ldots,s^K) \times \lambda_{c_i} =  \\\displaystyle
\prod_{\ell=1}^L \frac{\lambda_{c^\ell}}{\mu_{\S(\{c^1,\ldots,c^\ell\})}} 
\prod_{k=1}^K \frac{\mu_{s^k}}{\lambda_{\C(\{s^1,\ldots,s^k\})}} \times\,
\lambda_{c_i}; \\  
\\ \displaystyle
\sum_{\ell=0}^L  P^q(c^1,\ldots c^\ell, c_i, c^{\ell+1},\ldots,c^L,s^1,\ldots,s^K) \\ 
\qquad \times\, \mu_{\S(c_i)\backslash \S(\{c^1,\ldots,c^\ell\})} =
\\ \displaystyle
= \sum_{\ell=0}^L  
\prod_{j=1}^\ell \frac{\lambda_{c^j}}{\mu_{\S(\{c^1,\ldots,c^j\})}}
\times \frac{\lambda_{c_i}}{\mu_{\S(\{c_i,c^1,\ldots,c^\ell\})}} 
\\ \displaystyle
 \times \prod_{j=\ell+1}^L \frac{\lambda_{c^j}}{\mu_{\S(\{c_i,c^1,\ldots,c^j\})}}
\prod_{k=1}^K \frac{\mu_{s^k}}{\lambda_{\C(\{s^1,\ldots,s^k\})}} \\ 
\qquad \times \mu_{\S(c_i)\backslash \S(\{c^1,\ldots,c^\ell\})}.
\end{array}
\]

To show that the two expressions do indeed balance, we need to show that:
\begin{equation}
\label{eqn.show}
\begin{array}{l} \displaystyle
\prod_{\ell=1}^L \frac{1}{\mu_{\S(\{c^1,\ldots,c^\ell\})}} =
\sum_{\ell=0}^L  
\prod_{j=1}^\ell \frac{1}{\mu_{\S(\{c^1,\ldots,c^j\})}}   
\times  \frac{1}{\mu_{\S(\{c_i,c^1,\ldots,c^\ell\})}} \\\displaystyle
\times \prod_{j=\ell+1}^L \frac{1}{\mu_{\S(\{c_i,c^1,\ldots,c^j\})}}
\times \mu_{\S(c_i)\backslash \S(\{c^1,\ldots,c^\ell\})}
\end{array}
\end{equation}
which follows by  induction on $L$.  For $L=1$:
\[
\begin{array}{l} \displaystyle
\frac{1}{\mu_{\S(c_i)}}\frac{1}{\mu_{\S(c_i,c^1)}} \mu_{\S(c_i)} +
\frac{1}{\mu_{\S(c^1)}}\frac{1}{\mu_{\S(\{c_i,c^1\})}} \mu_{\S(c_i)\backslash \S(c^1)} 
\\
\\ \displaystyle
= \frac{1}{\mu_{\S(\{c_i,c^1\})}}\frac{\mu_{\S(c^1)} + \mu_{\S(c_i)\backslash \S(c^1)}}
{\mu_{\S(c_1)}}  = \frac{1}{\mu_{\S(c^1)}},
\end{array}
\]
and assuming that (\ref{eqn.show}) holds for $L-1$, we show that for  $L$:
\[
\begin{array}{l} \displaystyle
 \sum_{\ell=0}^L  
\prod_{j=1}^\ell \frac{1}{\mu_{\S(\{c^1,\ldots,c^j\})}}
\times \frac{1}{\mu_{\S(\{c_i,c^1,\ldots,c^\ell\})}} 
\\ \displaystyle
\times \prod_{j=\ell+1}^L \frac{1}{\mu_{\S(\{c_i,c^1,\ldots,c^j\})}}
\times \mu_{\S(c_i)\backslash \S(\{c^1,\ldots,c^\ell\})} 
\\ \displaystyle
= \sum_{\ell=0}^{L-1}  
\prod_{j=1}^\ell \frac{1}{\mu_{\S(\{c^1,\ldots,c^j\})}}
\times \frac{1}{\mu_{\S(\{c_i,c^1,\ldots,c^\ell\})}} 
\\ \displaystyle
\times \prod_{j=\ell+1}^{L-1}  \frac{1}{\mu_{\S(\{c_i,c^1,\ldots,c^j\})}} 
\frac{1}{\mu_{\S(\{c_i,c^1,\ldots,c^L\})}}
\times \mu_{\S(c_i)\backslash \S(\{c^1,\ldots,c^\ell\})} 
\\ \displaystyle
+ \prod_{j=1}^L \frac{1}{\mu_{\S(\{c^1,\ldots,c^j\})}}
\times \frac{1}{\mu_{\S(\{c_i,c^1,\ldots,c^L\})}} \times 
\mu_{\S(c_i)\backslash \S(\{c^1,\ldots,c^L\})} 
\\ \displaystyle
= 
\prod_{j=1}^{L-1} \frac{1}{\mu_{\S(\{c^1,\ldots,c^j\})}}\times\frac{1}{\mu_{\S(\{c_i,c^1,\ldots,c^L\})}}
\\ \displaystyle
+ \prod_{j=1}^L \frac{1}{\mu_{\S(\{c^1,\ldots,c^j\})}}
\times \frac{1}{\mu_{\S(\{c_i,c^1,\ldots,c^L\})}} \times 
\mu_{\S(c_i)\backslash \S(\{c^1,\ldots,c^L\})} 
\end{array}
\]
\[
\begin{array}{l} 
\displaystyle
= 
\prod_{j=1}^{L-1} \frac{1}{\mu_{\S(\{c^1,\ldots,c^j\})}}\times\frac{1}{\mu_{\S(\{c_i,c^1,\ldots,c^L\})}}
\left(1 +  \frac{\mu_{\S(c_i)\backslash \S(\{c^1,\ldots,c^\ell\})}}{\mu_{\S(\{c^1,\ldots,c^L\})}} \right)
\\ \displaystyle
= \prod_{j=1}^{L} \frac{1}{\mu_{\S(\{c^1,\ldots,c^j\})}}
\end{array}
\]

{\sc $\bullet$ Balance  of (iii) with (vi):}

For $s_j \not\in \S(\{c^1,\ldots,c^L\})$
\[
\begin{array}{l} \displaystyle
 P^q(c^1,\ldots,c^L,s^1,\ldots,s^K) \times \mu_{s_j}=  
 \\\displaystyle
\prod_{\ell=1}^L \frac{\lambda_{c^\ell}}{\mu_{\S(\{c^1,\ldots,c^\ell\})}} 
\prod_{k=1}^K \frac{\mu_{s^k}}{\lambda_{\C(\{s^1,\ldots,s^k\})}} \times\,
 \mu_{s_j}; \\  
\\ \displaystyle 
\sum_{k=0}^K  P^q(c^1,\ldots c^L,s^1,\ldots,s^k,s_j,s^{k+1},\ldots,s^K) \\ 
\qquad \times\, \lambda_{\C(s_j)\backslash \C(\{s^1,\ldots,s^k\})} 
\\ \displaystyle
= \sum_{k=0}^K  \prod_{\ell=1}^L \frac{\lambda_{c^\ell}}{\mu_{\S(\{c^1,\ldots,c^\ell\})}}
\prod_{i=1}^k \frac{\mu_{s^i}}{\lambda_{\C(\{s^1,\ldots,s^i\})}}
\frac{\mu_{s_j}}{\lambda_{\C(\{s_j,s^1,\ldots,s^k\})}}
\\ \displaystyle
\prod_{i=k+1}^K \frac{\mu_{s^i}}{\lambda_{\C(\{s_j,s^1,\ldots,s^i\})}}  \times\, \lambda_{\C(s_j)\backslash \C(\{s^1,\ldots,s^k\})}.
\end{array}
\]
To show that the two expressions do indeed balance, we need to show that:
\begin{equation}
\label{eqn.show2}
\begin{array}{l} \displaystyle
\prod_{k=1}^K \frac{1}{\lambda_{\C(\{s^1,\ldots,s^k\})}} =
\sum_{k=0}^K  
\prod_{i=1}^k \frac{1}{\lambda_{\C(\{s^1,\ldots,s^i\})}}   
\times  \frac{1}{\lambda_{\C(\{s_j,s^1,\ldots,s^k\})}} \\\displaystyle
\times \prod_{i=k+1}^K \frac{1}{\lambda_{\C(\{s_j,s^1,\ldots,s^i\})}}
\times \lambda_{\C(s_j)\backslash \C(\{s^1,\ldots,s^k\})}
\end{array}
\end{equation}
The proof of (\ref{eqn.show2}) is similar to the proof of (\ref{eqn.show})
\end{proof}

\section{Unique Path of the FCFS Infinite Directed  Matching Model}
\label{sec.uniquenss}
In  \ref{sec.uniquenss} and \ref{sec.reversal}, we prove properties of the FCFS directed matching of the i.i.d sequence of customer and server types $\ldots,z^1,z^2,\ldots$, where servers are only matched to previous customers, and of the Markov chain $U(n)$ of the leftover unmatched customers and servers.  We use the notation $\beta_j=\mu_{s_j}/(\rlambda+\rmu)$

\begin{proof}[Proof of Theorem \ref{thm.infbackmatch1}]
We prove the Theorem in several steps, requiring two lemmas and two propositions.
The two lemmas are pathwise results which do not depend on any probabilistic assumptions, and they prove  subadditivity and monotonicity.   Following that, Proposition \ref{thm.forward}  shows forward coupling, and Proposition \ref{thm.backcoupling} shows backward coupling.  The proof is then completed in a short paragraph.   This proof is very similar to the proof of Theorem 3 in \cite{adan2015reversibility}
\end{proof}

\begin{lemma}[Monotonicity]
\label{thm.monotonicity}
Consider a subsequence $z^1,\ldots,z^M$ of servers and customers, with all the possible FCFS matches of servers to previous customers.  Assume there are $K$  customers and $L$ servers left unmatched.  Consider now an additional element $z^0$ preceding $z_1$, and the complete FCFS matching of servers to previous customers of  $z^0,z^1,\ldots,z^M$. Then:

(i) If $z^0 = c^0$ is an additional customer,  
the  sequence $z^0,z^1,\ldots,z^M$ will have no more than $K+1$ customers and $L$  servers unmatched. 

(ii) If $z^0 = s^0$ is an additional server,  
the  sequence $z^0,z^1,\ldots,z^M$ will have exactly $K$ customers and $L+1$  servers unmatched. 
\end{lemma}
\begin{proof}
Statement (ii) is trivial, $s^0$ will be unmatched and all the other links in $s^0,z^1,\ldots,z^M$ will be unchanged from $z^1,\ldots,z^M$.

To prove (i), denote $A=(z^1,\ldots,z^M)$.
In the matching of $(c^0,A)$, if $c^0$ has no match, then all the other links in the matching are the same as in the matching of $A$, so the total number of unmatched customers is $K+1$ and unmatched servers is $L$.  
If $c^0$ is matched to a server $z^n=s^n$ and $s^n$ is unmatched in the matching of $A$  then $(c^0,s^n)$ is a new link, and all the other links in the matching of  $(c^0,A)$ are the same as in the matching of $A$, so  the total number of unmatched customers is $K$ and unmatched servers is $L-1$.

If $c^0$ is matched to $z^{n_1}$ and $z^{n_1}=s^{n-1}$ was matched to $z^{m_1}=c^{m_1}$ in the $A$ matching,  then $(c^0,s^{n_1})$ is a new link, and the link $(s^{n_1},c^{m_1})$ in the $A$ matching is disrupted.  We now look for a match for  $z^{m_1}=c^{m_1}$ in the matching of $(c^0,A)$.  Clearly, $c^{m_1}$ is not matched to any of $z^j=s^j,\,m_1<j<n_1$, since in  $A$ any such server was either matched to an earlier customer, and this link is still
there in the matching of $c^0,A$, or such a server is incompatible with $c^{m_1}$; otherwise $c_{m_1}$ could not have been matched to $s^{n_1}$ in $A$.  So $c^{m_1}$ will either remain unmatched, or it will be matched to some $z^{n_2}=s^{n_2}$, where $n_2>n_1$.  In the former case, all the links of the $A$ matching except $(s^{n_1},c^{m_1})$ remain unchanged in the matching of $(c^0,A)$, and so the numbers of unmatched items in $(c^0,A)$ is $K+1$ and $L$.  In the latter case, there are again two possibilities:  If $s^{n_2}$ is unmatched in the $A$ matching, it will now be matched to $c^{m_1}$ and the $(c^0,A)$ matching will have disrupted one link and added 2 links retaining all other links of the $A$ matching, so the numbers of unmatched items are $K$ and $L-1$.  If $s^{n_2}$ is matched to $z^{m_2}=c^{m_2}$ in the $A$ matching, then the link $s^{n_2},c^{m_2}$ is disrupted, and we now look for a match for  $c^{m_2}$ in the $(c^0,A)$ matching.  Similar to $c^{m_1}$, either $c^{m_2}$ remains unmatched, resulting in $K+1$ and $L$ unmatched items in the  
$(c^0,A)$ matching, or, by the same argument as before, $c^{m_2}$ will be matched to $s^{n_3}$, where $n_3>n_2$.  Repeating these arguments for any additional disrupted links, we conclude that  we either end up with one more link, so the number of unmatched items are $K$ and $L-1$, or we have the same number of links and the number of unmatched items are $K+1$ and $L$.
\end{proof}

\begin{lemma}[Subadditivity]
\label{thm.subadditivity}
Let $A'=(z^1,\ldots,z^m)$, $A''=(z^{m+1},\ldots,z^M)$ and let $A=(z^1,\ldots,z^M)$.  Consider the complete FCFS matching of servers to earlier customers in $A'$, in $A''$, and in $A$ and  let $K',K'',K$ be the number of unmatched customers and $L',L'',L$ be the number of unmatched servers in these three matchings.  Then $K \le K'+K''$ and $L\le L'+L''$.
\end{lemma}
\begin{proof}
Let $\hat{A}'=(\hat{z}^1,\ldots,\hat{z}^{K'+L'})$ be the ordered unmatched customers and servers from the complete FCFS matching of $A'$.  Then the FCFS matching of $(\hat{A}',A'')$ will have exactly the same ordered unmatched customers and servers as the FCFS matching of $A$.   We now construct the matching of $(\hat{A}',A'')$ in steps, starting with the matching of $(\hat{z}^{K'+L'},A'')$, next the matching of $(\hat{z}^{K'+L'-1},\hat{z}^{K'+L'},A'')$ and so on.  At each step, by Lemma \ref{thm.monotonicity}, if the added $z^j$ is a server the number of unmatched servers increases by 1, and the number of unmatched customers remains unchanged.  If the added $z^j$ is a customer the number of unmatched servers remains unchanged or decreases by 1, and  the number of unmatched customers  increases by 1 or remains unchanged.  It follows that the total number unmatched customers is $\le K'+K''$ and of unmatched servers is $\le L'+L''$.
\end{proof}

We assume  that the stability condition  (\ref{eqn.stability}) holds. 
By Theorem \ref{thm.embedded}, the augmented Markov chain of the infinite directed matching $U(n)$ is ergodic. Using the Kolmogorov extension theorem \cite{oksendal2003stochastic}, we may define (in a non-constructive way) a stationary version $U^*=(U^*(n))_{n=-\infty}^\infty$ of the Markov chain. Define also $U^{[k]}=(U^{[k]}(n))_{n= -k}^\infty$ the realization of the Markov chain that starts at $U^{[k]}(-k)=\emptyset$.

Our first task is to show forward coupling, namely that $U^*$ and $U^{[0]}$ coincide after a finite time $\tau$ with $E(\tau)<\infty$.  Following that we use standard arguments to show backward coupling and convergence to a unique matching.

\begin{proposition}[Forward coupling]
\label{thm.forward}
The two processes $(U^*(n))_{n=-\infty}^\infty$ and $(U^{[k]}(n))_{n= -k}^\infty$ will couple after a finite time $\tau$, with $E(\tau)<\infty$.
\end{proposition}

\begin{proof}  
Denote by $|u|$ the number of unmatched customers for any state $u$ of  process  $U$, we refer to it as the length of the state.   
Consider the sequence of times $0 \le M_0 < M_1 < \cdots < M_\ell,\cdots$ at which $U^*(M_\ell) = \emptyset$.  This sequence is infinite with probability 1, and $E(M_\ell) = E(M_0)+ \ell E(M_1-M_0) < \infty, \ell\ge 0$ by the ergodicity.  
Consider the state $u_0=U^{[0]}(M_0)$.  Then $|u_0| \le M_0$.   
By the monotonicity result of Lemma \ref{thm.monotonicity} and Lemma \ref{thm.subadditivity}, the states of $U^{[0]}$ satisfy
$|u_0| \ge |U^{[0]}(M_1)| \ge \ldots \ge |U^{[0]}(M_\ell)|$, i.e. the  length of the state of $U^{[0]}$ at the times $M_0,M_1,\ldots$ is non-increasing.  This is because each block of customers and servers in times between $M_{\ell-1}$ and $M_\ell$ on its own has 0 unmatched.  Furthermore, if the first unmatched customer in $U^{[0]}(M_\ell)$ is $c_i$, and 
the following item in the infinite sequence of customers and servers, $z^{M_\ell+1}$ is $s_j\in S(c_i)$, then $M_{\ell+1} = M_\ell+1$, and $|U^{[0]}(M_\ell+1)| = |U^{[0]}(M_\ell)|-1$.  This will happen with probability $\ge \delta =\min(\beta_1,\ldots,\beta_J)$.  
Hence, there will be coupling after at most  $\sum_{j=1}^{|u_0|} L_j $ perfect matching blocks of $U^*$, where $L_j$ are i.i.d. geometric random variables with $\delta$ probability of success.  So coupling occurs almost surely, and the coupling time $\tau$ satisfies $E(\tau) \le  E(M_0) \left( \frac{1}{\delta} E(M_1-M_0) \right)$.

The proof for $U^{[k]}$ is the same.
\end{proof}

Note that once $U^{[k]}$ and $U^*$ couple, they stay together forever.
We now need to show backward coupling.   

\begin{proposition}[Backward coupling]
\label{thm.backcoupling}
Let $U^*$ be the stationary version of the Markov chain $U(\cdot)$, and let  $U^{[-k]}$ be the process starting empty at time $-k$.    Then $\lim_{k\to\infty} U^{[-k]}(n) = U^*(n)$ for all $-\infty < n < \infty$ almost surely.
\end{proposition}
\begin{proof}
The statement of almost surely refers to the measure of the infinite sequences  $\ldots,z^{-1},z^0,z^1,\ldots$.

Define $T_k = \inf \{n\ge -k: U^{[-k]}(n) = U^*(n) \}$.  By the forward coupling Proposition \ref{thm.forward}, we get that $T_k$ is almost surely finite. 
Let $\hat{T}_K = \max_{0\le k \le K} T_k$.  It is  $\ge 0$, and is also almost surely finite 
for any $K$.  $\hat{T}_K$ is the time at which all the processes starting empty at time $-k$, where $0\le k\le K$, couple with $U^*$, and remain merged forever.
Define the event $E_K = \{ \omega : \forall \ell \ge 0,\, U^{[-\ell]}(\hat{T}_K) = U^*(\hat{T}_K) \}$, in words, those $\omega$ for which the process starting empty at any time before $0$, will merge with $U^*$ by time $\hat{T}_K$.  We claim that $P(E_K)>0$.  We evaluate $P(\overline{E}_K)$.  
For any fixed $\ell \geq 0$, call $E_{\ell,K}$ the event that $U^{[-\ell]}$ couples with $U^*$ by time $\hat{T}_K$. 
We have 
$E_K = \bigcap_{\ell \geq 0} E_{\ell,K} = \bigcap_{\ell > K} E_{\ell,K}$ (by definition of $\hat{T}_K$, $E_{\ell,K}$ is always true for $\ell\leq K$, so we only need to consider $\ell>K$), so $\overline{E_K} = \bigcup_{\ell > K} \overline{E_{\ell,K}}$.   

The event $\overline{E_{\ell,K}}$ will happen if starting at the last time prior to $-K$ at which the process $U^{[-\ell]}$ was empty,  the next time that it is empty is after time $0$.  
The reason for that is that otherwise the process $U^{[-\ell]}$ reaches state $\emptyset$ at some time  $k \in [-K,0]$ and from that time onwards it is coupled with  $U^{[-k]}$, and will couple with $U^*$ by time $\hat{T}_K$.   

Define for $\ell>K$, $D_\ell =\{\omega : O^{[-\ell]}(m) \ne \emptyset \mbox{ for all} -\ell < m \le 0\}$.  Clearly by the above, $\bigcup_{\ell > K} \overline{E_{\ell,K}} \subseteq \bigcup_{\ell > K} D_\ell$.  
Let $\tau$ denote the recurrence time of the empty state.  Then: 
\[
P(\overline{E_K}) = P(\bigcup_{\ell > K} \overline{E_{\ell,K}}) \le P(\bigcup_{\ell > K} D_\ell) \le 
\sum_{\ell > K} P(\tau > \ell).
\]
By the ergodicity $\sum_{l=0}^\infty P(\tau> l) =E(\tau)<\infty$.  Hence we have that $P(\overline{E_K}) \to 0$ as $K\to\infty$, and therefore  $P(E_K)>0$ for large enough $K$, and 
$P(E_K) \to 1$ as $K\to \infty$.  Note also that $E_K \subseteq E_{K+1}$.

Define now $\hat{T}= \sup_{k\ge 0}  T_k$.  We claim that $\hat{T}$ is finite a.s.  Consider any $\omega$. Then by $P(E_K) \to 1$ as $K\to \infty$ and by the monotonicity of $E_K$, almost surely for this $\omega$ there exists a value $\ell$ such that $\omega \in E_\ell$.  But if $\omega \in E_l$, then $\hat{T}(\omega) \le \hat{T}_\ell < \infty$.   

So,  all processes starting empty before time $0$ will couple with $U^*$ by time $\hat{T}$.  
By the stationarity of the sequences $(z^n)_{n=-\infty}^\infty$ and of $U^*$,   
we then also have that all processes $U^{[-k]}(n)$ starting empty before $-k$ will couple with $U^*$ by time 0, if $k\ge \hat{T}$.  
Hence using the Loynes' scheme of starting empty at $-k$ and letting $k\to\infty$ the constructed process will merge with $U^*$ at time 0.  But the same argument holds not just for 0, but for any negative time $-n$.   Hence $U^{[-k]}$ and $U^*$ couple at $-n$ (and stay coupled) for any $k > n +  \hat{T}$.  This completes the proof.
\end{proof}

\begin{proof}[End of proof of Theorem \ref{thm.infbackmatch1}]
We saw that $\lim_{k\to\infty} U^{[-k]}(n)=U^*(n)$ for all $n$ almost surely.  Each process $U^{[-k]}(n)$ determines matches uniquely for all $n> - k$, so if we fix $n$, matches from $n$ onwards are uniquely determined by $\lim_{k\to\infty} U^{[-k]}(n)$.  Hence $(U^*(n))_{n=\-\infty}^\infty$ determines for every   customer $z^n=c^n$ his match, uniquely, almost surely.  This proves the  theorem. 
\end{proof}

\section{Time Reversal of the FCFS Infinite Directed  Matching Model}
\label{sec.reversal}

To prove Theorem \ref{thm.infbackmatch2}, we consider blocks of the form $z^1,\ldots,z^n$ such that  all the customers in the block are matched to servers further in the sequence, we refer to those as perfect blocks.  We first show in Lemma \ref{thm.reverse} that the exchange transformation implies time reversal in each block.  Next in Lemma \ref{thm.blocks} we show that perfect blocks and their reversal have the same probability. The proof of the theorem then follows by considering Palm measure and time stationary measure of the exchanged sequence.

\begin{lemma}
\label{thm.reverse}
Let $z^1,\ldots,z^n$ be a perfect block of customers and servers, and let $\tz^1,\ldots,\tz^n$ be the block obtained from $z^1,\ldots,z^n$ by the exchange transformation.  Then 
$\tz^n,\ldots,\tz^1$ is also a perfect block.  In other words, if we have a block where all the customers are matched FCFS to servers ahead of them in the sequence, and we exchange the positions of matched pairs of customers and servers and retain the links, then the resulting matching is FCFS of servers to ahead of them in the sequence in reversed time.
\end{lemma}

\begin{proof}
Consider the sequence $\tz^n,\ldots,\tz^1$, and assume that $\tz^k = c_i$ is coupled in the exchanged sequence to $\tz^l = s_j$, with $l$ ahead of $k$ in the reversed sequence, i.e.
$l<k$.  Then we look at $\tz^{l'}=s_{j'}$ with $l <l' <k$. There are two possibilities:  If $\tz^{l'}$ is unmatched, then $z^{l'}=\tz^{l'}$ because it was not exchanged.  Hence in the original sequence $z^l=c_i$ precedes $z^{l'}=s_{j'}$ precedes $z^k=s_j$.  But then $s_{j'}$ must be  incompatible with $c_i$, or else $z^l=c_i$ would have  matched with $z^{l'}=s_{j'}$ in the original sequence.  
\begin{figure}[htbp]
   \centering
\includegraphics[width=2.5in]{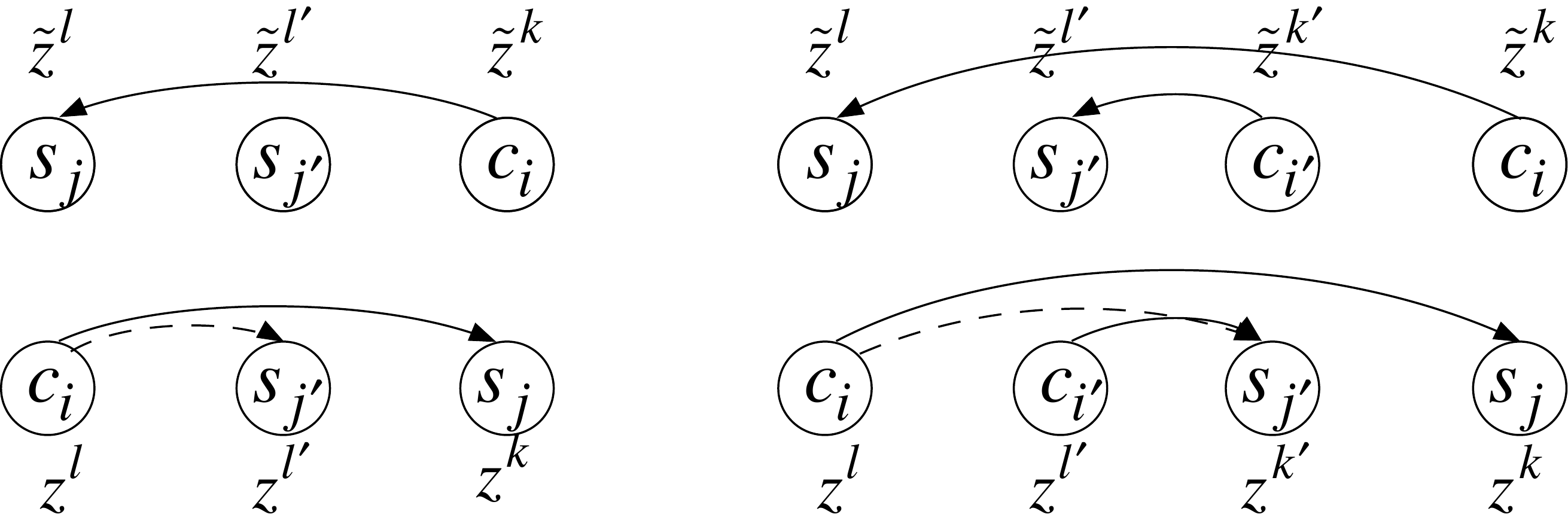}
   \caption{Illustration of the proof of time reversal}
   \label{fig.reversalproof}
\end{figure}
The other possibility is that $\tz^{l'}$ has been matched and exchanged with $\tz^{k'}=c_{i'}$.    Assume now that $s_{j'}$ is compatible with $c_i$.  Then we must show that $k' > k$.  Assume to the contrary $l<l' < k' <k$.  Then in the original sequence 
$\tz^l=c_i$ precedes $\tz^{l'}=c_{i'}$ precedes $\tz^{k'}=s_{j'}$ precedes $\tz^k=c_i$.  But then $z^l=c_i$ would have matched with $z^{k'}=s_{j'}$ in the original sequence.
This completes the proof.
The proof is illustrated in  Figure \ref{fig.reversalproof}:
\end{proof}

\begin{lemma}
\label{thm.blocks}
Consider the FCFS directed matching of $\ldots,z^{-1},z^0,z^1,\ldots$, and let  $z^{m+1},\ldots,z^{m+M}$   be the block of  customers and servers in positions $[m+1,m+M]$. Then the conditional probability of observing these values, conditional on the event that the FCFS directed matching of these values is a perfect match is: 
\begin{eqnarray*}
&& P\Big(\big(z^{m+1},\ldots,z^{m+M}\big)\Big| \big(z^{m+1},\ldots,z^{m+M}\big) \mbox{ has perfect match}\Big)  \\
&& \quad = \kappa_M    \prod_{i=1}^{I}  {\alpha_{c_i}}^{\# c_i} \prod_{j=1}^{J} {\beta_{s_j}}^{\# s_j}
\end{eqnarray*}
where $\kappa_M$ is a constant that may depend on $M$, and $\# c_i$, $\# s_j$ count the number of type $c_i$ customers and type $s_j$ servers in the block.
\end{lemma}
\begin{proof}
The conditional probability is calculated using Bayes formula:
\begin{eqnarray*}
&& \quad  P(\mbox{seeing $z^{m+1},\ldots,z^{m+M}$ }| \mbox{ having a perfect match}) \\
&& \quad = \frac{P(\mbox{having a perfect match }| 
\mbox{seeing $z^{m+1},\ldots,z^{m+M}$ })}
{P( \mbox{having a perfect match of length $M$})} \\
&& \qquad \times P(\mbox{seeing $z^{m+1},\ldots,z^{m+M}$ }) \\
&& \quad = \kappa_M \times \mathbf{1}_{\{(z^{m+1},\ldots,z^{m+M}) \mbox{ is a perfect match}\}} 
 \times  \prod_{i=1}^{I}
{\alpha_{c_i}}^{\# c_i} \prod_{j=1}^{J} {\beta_{s_j}}^{\# s_j}
\end{eqnarray*}
where $\kappa_M=1\big/ P( \mbox{having a perfect match of length $M$})$.
\end{proof}

\begin{corollary}
\label{thm.symmetric} 
Consider the FCFS directed matching of $\ldots,z^{-1},z^0,z^1,\ldots$,  let  $z^{m+1},\ldots,z^{m+M}$  be the block of  customers and servers in positions $[m+1,m+M]$, which has perfect matching, and let $\tz^{m+1},\ldots,\tz^{m+M}$ be its exchange transformation.  Replace $z^{m+1},\ldots,z^{m+M}$ by $\tz^{m+M},\ldots,\tz^{m+1}$. Then $\tz^{m+1},\ldots,\tz^{m+M}$ will be a perfectly matched block in the new directed matching of the complete sequence, and:
\[
P(\tz^{m+1},\ldots,\tz^{m+M}) = P(z^{m+1},\ldots,z^{m+M}).
\]
\end{corollary}
\begin{proof}
That $\tz^{m+1},\ldots,\tz^{m+M}$ is a FCFS directed perfectly matched  block follows from Lemma \ref{thm.reverse} and that $P(\tz^{m+M},\ldots,\tz^{m+1}) = P(z^{m+1},\ldots,z^{m+M})$ follows from Lemma \ref{thm.blocks}.
\end{proof}

We now assume that the system is ergodic, i.e. the Markov chains $U(n)$ is ergodic (which holds if and only if the stability condition (\ref{eqn.stability}) holds).
We have shown in Theorem \ref{thm.infbackmatch1} that for i.i.d  sequences of servers and customers $\ldots,z^{-1},z^0,z^1,\ldots$, under ergodicity,  there exists a.s. a unique FCFS directed matching, which corresponds to the stationary version of the Markov chain $U$ (generated by the Loynes' construction).    We now define the following augmented Process, with paths $\fp$ where the state consists of $\fp(n)=(U(n),z^n,v_n)$, where $v_n$ records the location of the element that is matched with $z^n$, i.e. if $z^n=c_i$ is matched to $z^m=s_j$ where $m>n$ then $v_n=m$, if $z^n=s_j$ is matched to $z^m=c_i$ where $m<n$ then $v_n=m$, and if $z^n=s_j$ is unmatched then $v_n=n$.
The path $\fp$ is uniquely determined by the i.i.d. sequence $\ldots,z^{-1},z^0,z^1,\ldots$.  We denote by $\fP$ be the probability distribution of the paths $\fp$.   We now define paths $\psi\fp$ by the exchange transformation followed by time reversal:  To $\fp(n)=(U(n),z^n,v_n)$  we define $\psi\fp(-n)=(\tilde{U}(-n),\tz^{-n},v_{-n})$ where $\tz^{-n} = z^{v_n}$, $\tilde{v}_{-n}=-v_n$, and $\tilde{U}(m)$ for every $m$ is defined as the customers that are unmatched, and the servers that are unmatched starting from the last unmatched customer, in the sequence of $\tz^r,\tilde{v}_r$, $r\le m$ (i.e. it behaves just like $U(n)$, but is obtained from the given matching that is already determined by $\tz,\tilde{v}$).  For every path $\fp$ there is a transformed path $\psi\fp$.  

We denote by $\psi\fP$ the distribution of the transformed paths $\psi\fp$.   
Our goal is to show that $\psi\fP=\fP$.

Let $\mathfrak{P}^0$ be the Palm version of the measure $\mathfrak{P}$, with respect to the state $\emptyset$ of $U$,  that is, $\mathfrak{P}^0$ is the law of $(U(n),z^n,v_n)$ conditioned on the event $\{U(0)=\emptyset\}$. 
A realization of a process of law $\mathfrak{P}^0$ can be obtained by considering a bi-infinite sequence of  perfectly matched blocks of i.i.d customers and servers.  Denote by $O(m),\,m=0,\pm 1, \pm 2,\ldots$ a sequence of i.i.d. minimal perfectly matched blocks.  Then the resulting paths of these will have the distribution $\fP^0$.  
Now perform the exchange transformation on this sequence, followed by time reversal, and let $\psi\fP^0$ be the probability distribution of the result. To $\psi\fP^0$, correspond $\psi\fP$ which is the stationary version of $\psi\fP^0$. 
Now according to Lemma \ref{thm.reverse}, we have $\psi\fP^0 =\fP^0$. Since $\psi\fP$ is the stationary version of $\psi\fP^0$ and since $\fP$ is the stationary version of $\fP^0$, we deduce that $\psi\fP=\fP$. 

The key  argument to show this, is the link between time-stationarity and event-stationarity.  For general background on Palm calculus, see for instance Chapter 1 in \cite{BaccelliB:03}.   So we obtain the following result. 

\begin{proposition}
Consider a FCFS directed matching model under the stability condition (\ref{eqn.stability}). 
Let  $\ldots,z^{-1},z^0,z^1,\ldots$ be the independent i.i.d. sequence of customers and servers, with the  unique FCFS matching between them.  Then the exchanged sequence 
$\ldots,\tz^{-1},\tz^0,\tz^1,\ldots$ is also  i.i.d. of the same law.
Furthermore, the FCFS directed matching for the exchanged sequence in reversed time, consists of the same links as the matching of the original sequence.
\end{proposition}

\begin{proof}[Proof of Theorem \ref{thm.infbackmatch2}]
That $\ldots,\tz^{-1},\tz^0,\tz^1,\ldots$ is an i.i.d. sequence follows from the identity of $\psi\fP$ and  $\fP$.  That the Loynes' construction in reversed time will use the same links follows, since the links of $\ldots,\tz^{-1},\tz^0,\tz^1,\ldots$ give a set of links which are the  FCFS directed matchings in reversed time between  $\ldots,\tz^{-1},\tz^0,\tz^1,\ldots$,  and by Theorem \ref{thm.infbackmatch1} this matching is unique.
\end{proof}

\section{Proof of Theorem \ref{thm.N-comparison}, on Comparison of `N'-system Policies}
\label{sec.N-comparison}
We will prove the stronger result:
\begin{proposition}
\label{thm.N-comparisontransient}
Starting from an empty system, for every $t>0$ the following holds:
\begin{equation}
\label{eqn.comparison}
N_1^r(t) + N_2^r(t) \ge_{ST}  N_1^q(t) + N_2^q(t) -1. 
\end{equation}
\end{proposition}
This will by ergodicity prove that Theorem \ref{thm.N-comparison} holds.
\begin{proof}[Proof of Proposition \ref{thm.N-comparisontransient}]
We consider the coupled systems.  We refer to the system under the Redundancy policy as system-r and to the system under FCFS-ALIS as system-q.  We assume that system-q and system-r share a sequence of events $z^1,z^2,\ldots$, where $z^n$ take the values $c_1,c_2,s_1,s_2$ to denote arrival of either type or potential service completion of either type, and $z^n$ are i.i.d. with probabilities:  
\[
P(Z^n=c_1)=\frac{\lambda_1}{\rlambda+\rmu},\quad P(Z^n=c_2)=\frac{\lambda_2}{\rlambda+\rmu},\quad P(Z^n=s_1)=\frac{\mu_1}{\rlambda+\rmu},\quad P(Z^n=s_2)=\frac{\mu_2}{\rlambda+\rmu}.
\]
To simplify notation we let $N^h(T)=N^h_1(T)+N^h_2(T)$ where $h$ is either $r$ or $q$, and $T=1,2,\ldots$ is the $T$th event, and we are always regarding the state at event $T$ to be the state after the $T$th event happened.

We define the state of the two systems, at event $T$ by:
\[
Y^h(T) = (y^h_{-1}(T),y^h_{0}(T),y^h_1(T),\ldots,y^h_L(T))
\]
where 
\[
\begin{array}{ll}
y^h_{-1}(T) = n & \mbox{if $z^n$ is served at server 1 after event $T$}, \\
y^h_0(T) = n & \mbox{if $z^n$ is  served at server 2 after event $T$}, \\
y^h_j(T) = n & \mbox{if $z^n$ is  waiting in position $j$ in queue after event $T$, $j=1,\ldots,L$}, \\
y^h_{-1}(T) =0 & \mbox{if server 1 is idle after event $T$}, \\
y^h_0(T) =0 & \mbox{if server 2 is idle after event $T$}, \\
\end{array}
\]
where $h$ is either $r$ or $q$.  
This state simply tells us which customers are in service, and which customers are in the ordered queue.  Note that $y^h_{-1}(T)=n$ then $z^n=c_1$, because server 1 only serves customers of type 1.  Also, in system-r it is possible to have $y^r_{-1}(T)=y^r_0(T)=n$ if  $z^n=c_1$ and this customer is served simultaneously by both servers.

We now define at event $T$ the location in the systems for customers of type 2.  If $z^n=c_2,$ for $n\le T$, then
\[
L_T^h(n) = \left\{ \begin{array}{ll}
j & \mbox{if } y^h_j(T) = n \\  -1 & \mbox{if customer $n$ has  departed before or at event $T$} 
\end{array} \right.
\]
To prove the proposition  we will prove the following:

{\bf Statement:}  For all $T$ and $n\le T$ such that $z^n=c_2$,  
\begin{equation}
\label{eqn.statement}
L_T^r(n) \ge L_T^q(n) -1.
\end{equation}
that is, after every event, the difference in location of a customer of type 2 in system-q minus its position in system-r is always less then or equal to 1. 

We prove this statement by induction on $T$.  It is certainly true for $T=0$, when the systems are empty.  We now assume that the statement holds at all the events $T'<T$.  We need to check that it still holds at event $T$.  We will check it for each of the four types of events.

{\em Case 1:}  event $T$ is an arrival of a customer of type 1.

An arrival of a customer of type 1 does not change the location of any of the type 2 customers so the statement holds at $T$.

{\em Case 2:}  event $T$ is a potential service completion by server 2.

If prior to the event server 2 in system-q was busy, there is a departure from system-q, so positions of all type 2 customers in system-q decrease by 1 and the statement remains true for $T$.

If prior to the event server 2 in system-q was idle, than before and after that event  system-q contained no customers of type 2,  and therefore the statement holds trivially.

{\em Case 3:}  event $T$ is an arrival of a customer of type 2.

We will show that the induction hypothesis implies that before event $T$, the total number waiting for service in system-q is at most one more than in system-r.  This implies that when in event $T$ a customer of type 2 arrives his location at the end of the queue in system-q and in system-r satisfies the statement.

We will first show that the induction hypothesis implies that before event $T$, for $T'<T$, $N^q(T')-N^r(T')\le 1$.
To show that, we will show that if for some $T'-1<T-1$ we already have $N^q(T'-1)-N^r(T'-1) = 1$, then, by the induction hypothesis (\ref{eqn.statement}) for $T'$, the difference cannot grow at the time of event $T'$.

The difference can only grow if at the next potential service completion a customer leaves from system-r but does not leave from system-q.  If the potential service completion is of server 2, then if no customer leaves from system-q it means that $N^q(T'-1)\le 1$ and if a customer leaves from system-r it means that $N^r(T'-1) \ge 1$, which contradicts $N^q(T'-1)-N^r(T'-1) = 1$.

  We now need to consider potential service completions at server 1.  We need to check two situations:

\begin{compactitem}[-]
\item
 Prior to  event $T'$, server 1 and server 2 in system-r are serving a customer of type 1, and server 1 in system-q is idle.  In this case both servers in system-r are serving the same customer.   In system-q at that time there must be at least 2 customers, so there must be $k\ge1$ customers waiting, and all of them are of  type 2, since server 1 is idle.  Since the difference $N^q(T'-1)-N^r(T'-1) = 1$, we must have that in system-r there are only $k-1$ customers waiting.  But then one of the type 2 customers waiting in system-q must have already departed in system-r, which violates the induction hypothesis.
\item
Prior to  event $T'$, in system-r server 1 is serving a customer of type 1 and server 2 is serving a customer of type 2,  and server 1 in system-q is idle.  
 In system-q at that time there must be $k\ge2$ customers waiting, and all of them are of  type 2.  Since the difference $N^q(T'-1)-N^r(T'-1) = 1$, we must have that in system-r there are only $k-2$ customers waiting, so there is a total of no more than $k-1$ customers of type 2 in system-r.  But then (as in the previous situation) one of the type 2 customers waiting in system-q must have already departed in system-r, which violates the induction hypothesis.
\end{compactitem}
 
 We have then shown that for $T'<T$, total numbers in the systems satisfy  $N^q(T')-N^r(T')\le 1$.   We now show that for $T'<T$ the total number waiting in system-q minus total waiting in system-r cannot be more than 1.   We will show that this holds at event  $T'$ given the induction hypothesis for $T'<T$, and also given that $N^q(T')-N^r(T') \le 1$, which we just proved.   Assume to the contrary that at $T'$, there are $k+2$ customers waiting in system-q, and at most $k$ customers waiting in system-r.  Then by $N^q(T')-N^r(T') \le 1$,  there  must be 2 customers in service in system-r, and in system-q server 1 is idle, and all $k+2$ customers waiting in system-q must be of type 2.  But then one of the type 2 customers waiting in system-q must have already departed in system-r, which violates the induction hypothesis.

{\bf Note:} In our proof of Case 3, we have shown that:
\begin{equation}
\label{eqn.induction}
L_T^r(n) \ge L_T^q(n) -1  \quad \Longrightarrow  N^q(T) - N^r(T) \le 1
\end{equation}

{\em Case 4:}  Event $T$ is a potential service completion by server 1.

There are three situations to consider:

{\em Case 4-A:}  Prior to event $T$ server 1 in system-r is idle.  In that case there is no departure from system-r, and the statement holds at $T$.  

{\em Case 4-B:}  Prior to event $T$ server 1 in system-r  is serving a customer of type 1, and server 2 in system-r is serving a customer of type 2.  We only need to consider the following situation, for an arbitrary customer $z^n=c_2$.     Prior to event $T$ customer $z^n$ is waiting for service both in system-q and in system-r,  and $L^q_{T-1}(n) - L^r_{T-1}(n) =1$.  
Otherwise the locations of $z^n$ in the two system after event $T$ cannot differ by more than 1.  Furthermore, we only need to consider the situation where there is at least one customer $z^m=c_1$ with $m<n$ waiting in system-r, and  all customers waiting in system-q before $z^n$ are of type 2.  Otherwise, if there is no type 1 customer waiting in system-r, the location of $z^n$  in system-r will not change at $T$, and if there is a customer of type 1 waiting in system-q in a location before $z^n$ then the location of customer $z^n$ in system-q will decrease by 1 at $T$.

So we now assume that $L^q_{T-1}(n) - L^r_{T-1}(n) =1$ and there is at least one customer $z^m=c_1$ with $m<n$ waiting in system-r, and  all customers waiting in system-q before $z^n$ are of type 2.
In that case, there are $L^q_{T-1}(n) -1$ type 2 customers in queue before $z^n$ in system-q, and at most 
$L^q_{T-1}(n) -3$ customers of type 2 in queue before $z^n$ in system-r.  We already observed that in that case at least one of the type 2 customers waiting in system-q prior to $T$ must have already departed from system-r, which contradicts the induction hypothesis.

{\em Case 4-C:}  Prior to event $T$ server 1 and server 2 in system-r  are both serving simultaneously the same customer of type 1.  We now need to consider two sub-cases.

{\em Case 4-C(i):} Prior to event $T$, all the customers waiting in system-q are of type 2.
Let these customers be $z^{n_1}=c_2,z^{n_2}=c_2,\ldots,z^{n_k}=c_2$.  Their locations in sytem-q are $L_{T-1}^q(n_j)=j,\,j=1,\ldots,k$.   Then $z^{n_1}$ must also be present in system-r, or else the induction hypothesis is violated.  But then also all of $z^{n_j}=c_2,\,j=1,\ldots,k$ must be present in system-r, by FCFS.  Furthermore, $z^{n_1}$ is not in serivce in systme-r, so $L_{T-1}^r(n_1)\ge 1 = L_{T-1}^q(n_1)$,  and because system-q has no waiting customers of type 1, while system-r may have customers  of type 1, it follows that for $j=2,\ldots,k$
\begin{equation}
\label{eqn.diff}
L_{T-1}^r(n_j) - L_{T-1}^q(n_j) \ge L_{T-1}^r(n_1) - L_{T-1}^q(n_1) + \mbox{number of type 1 customers in system-r, between $z^{n_1}$ and $z^{n_j}$}
\end{equation}
If at event $T$ only one customer of type 1 enters service simultaneously in both servers, the statement (\ref{eqn.statement}) still holds at $T$.  If two customers enter service,  then they must be the first waiting customer which is of type 2, and the first waiting customer of type 1.   Then in system-r the locations of all customers waiting before the first customer of type 1 decrease by 1, and the locations of all customers waiting after the first customer of type 1 decrease by 2.
It  is then easy to see by (\ref{eqn.diff}) that after event $T$ the statement (\ref{eqn.statement})  still holds.

{\em Case 4-C(ii):}  Prior to event $T$, there is at least one waiting customer of type 1  in system-q.  

Consider customer $z^n=c_2$ that is waiting in system-q, prior to event $T$.  Note first that because  server 2 is serving a customer of type 1 in system-r, by the induction hypothesis, 
customer $z^n$ must also be waiting in system-r, and   $L^q_{T-1}(n) - L^r_{T-1}(n) \le 1$.
We need to show that after event $T$ we still have $L^q_{T}(n) - L^r_{T}(n) \le 1$.
The difficulty now is that at event $T$  two customers may enter service in system-r.   
We consider several possibilities.
\begin{compactitem}[-]
\item
Consider $L^q_{T-1}(n) - L^r_{T-1}(n) \le -1$.  Then at event $T$, even if in system-r two customers waiting in locations before $z^n$ enter service, and in system-q no customer in location before $z^n$ enters service, we still have $L^q_{T}(n) - L^r_{T}(n) \le 1$.
\item
Consider $L^q_{T-1}(n) - L^r_{T-1}(n) =0$.  Assume that at event $T$, in system-r, two customers waiting in locations before $z^n$ enter service, and in system-q no customer waiting in location before $z^n$ enters service, and  let $k=L^q_{T-1}(n)=L^r_{T-1}(n)$.  Then prior to event $T$,  in system-q all $k$ customers waiting up to location $k$ are of type 2, while in system-r, at least one is of type 1.  So one of the customers waiting in system-q must have already departed from system-r, which contradicts the induction hypothesis.  So this assumption leads to a contrdiction, and this confirms that if $L^q_{T-1}(n) - L^r_{T-1}(n) =0$, we still have $L^q_{T}(n) - L^r_{T}(n) \le 1$.
\item
Consider $L^q_{T-1}(n) - L^r_{T-1}(n)  =1$, let $k=L^q_{T-1}(n)$ so $L^r_{T-1}(n)=k-1$.  
We need to consider three situations:  

: Assume that prior to event $T$, in system-q there is no customer of type 1 waiting in a location before $z^n$. Then in system-q there are $k-1$ customers of type 2 in locations before $z^n$, and in system-r there are no more than $k-2$ customers of type 2 in locations before $z^n$, and in system-r there is no customer of type 2 in service,  which contradicts the induction hypothesis.

: Assume that prior to event $T$, in system-q there is at least one customer of type 1 waiting in a location before $z^n$, and in system-r there is no customer of type 1 waiting in a location before $z^n$. Then at event $T$ exactly one customer from location before $z^n$ enters service in both systems, so we still have $L^q_{T}(n) - L^r_{T}(n) \le 1$.

: Assume that prior to event $T$, there is at least one customer of type 1 waiting in a location before $z^n$ in both systems.  Then the customers in locations after $z^n$ must be the same customers in both systems, by FCFS.  In system-q there are 2 customers in service, since there is at least one customer of type 1 in the queue.  In system-r there is only one customer that is served simultaneously be both servers.  In system-q there are $k-1$ customers waiting in locations before customer $z^n$, and in system-r there are only $k-2$ customers waiting in locations before customer $z^n$.  But in that case, $N^q(T-1) = N^r(T-1) +2$.  This is in contradiction to our proof in Case C that by (\ref{eqn.induction})  $N^q(T-1) \le N^r(T-1) +1$.  
\end{compactitem}

This completes the proof of statement (\ref{eqn.statement}).

We now have from (\ref{eqn.induction}), since we have shown that statement (\ref{eqn.statement}) holds for all $T$, that also (\ref{eqn.comparison}) holds for all $T$. This completes the proof. 
\end{proof}

{\noindent \bf References}

\bibliographystyle{plain}
\bibliography{FCFSparBibliography} 

\end{document}